\documentclass[sn-mathphys-num]{sn-jnl}

\usepackage{graphicx}%
\usepackage{multirow}%
\usepackage{amsmath,amssymb,amsfonts}%
\usepackage{amsthm}%
\usepackage{epsfig}
\usepackage{epstopdf}
\usepackage{mathrsfs}%
\usepackage[title]{appendix}%
\usepackage{xcolor}%
\usepackage{textcomp}%
\usepackage{manyfoot}%
\usepackage{booktabs}%
\usepackage{algorithm}%
\usepackage{algorithmicx}%
\usepackage{algpseudocode}%
\usepackage{listings}%
\usepackage{cases}%
\graphicspath{{fig/}}
\usepackage{tikz}

\theoremstyle{thmstyleone}%
\newtheorem{theorem}{Theorem}
\newtheorem{proposition}[theorem]{Proposition}%

\theoremstyle{thmstyletwo}%
\newtheorem{remark}{Remark}%
\newtheorem{assumption}{Assumption}
\newtheorem{lemma}{Lemma}
\newtheorem{corollary}{Corollary}%

\theoremstyle{thmstylethree}%
\newtheorem{definition}{Definition}%

\begin{document}

\title[Convergence of the majorized PAM method with subspace correction for low-rank composite factorization model]{Convergence of the majorized PAM method with subspace correction for low-rank composite factorization model}


\author[1]{\fnm{Ting} \sur{Tao}}\email{taoting@fosu.edu.cn}

\author[2]{\fnm{Yitian} \sur{Qian}}\email{yitian.qian@polyu.edu.hk}

\author*[3]{\fnm{Shaohua} \sur{Pan}}\email{shhpan@scut.edu.cn}

\affil[1]{\orgdiv{School of Mathematics}, \orgname{Foshan University}, \orgaddress{\street{Zhangcha Street}, \city{Foshan}, \postcode{528000}, \state{Guangdong}, \country{China}}}

\affil[2]{\orgdiv{Department of Data Science and Artificial Intelligence}, \orgname{The Hong Kong Polytechnic University}, \orgaddress{\street{Hung Hom Street}, \city{Hong Kong}, \postcode{999077}, \country{China}}}

\affil[3]{\orgdiv{School of Mathematics}, \orgname{South China University of Technology}, \orgaddress{\street{Wushan Street}, \city{Guangzhou}, \postcode{510641}, \state{Guangdong}, \country{China}}}


\abstract{This paper focuses on the convergence certificates of the majorized proximal alternating minimization (PAM) method with subspace correction, proposed in \cite{TaoQianPan22} for the column $\ell_{2,0}$-norm regularized factorization model and now extended to a class of low-rank composite factorization models from matrix completion. The convergence analysis of this PAM method becomes extremely challenging because a subspace correction step is introduced to every proximal subproblem to ensure a closed-form solution. We establish the full convergence of the iterate sequence and column subspace sequences of factor pairs generated by the PAM, under the KL property of the objective function and a condition that holds automatically for the column $\ell_{2,0}$-norm function. Numerical comparison with the popular proximal alternating linearized minimization (PALM) method is conducted on one-bit matrix completion problems, which indicates that the PAM with subspace correction has an advantage in seeking lower relative error within less time.}

\keywords{Low-rank factorization models, alternating minimization methods, subspace correction, global convergence, KL property}


\pacs[MSC Classification]{49J52,  90C26,  49M27}

\maketitle

\section{Introduction}\label{sec1}
 Let $\mathbb{R}^{n\times m}$ denote the space of all $n\times m$ real matrices, equipped with trace inner product $\langle\cdot,\cdot\rangle$ and its induced Frobenius norm $\|\cdot\|_F$. Fix any $r\in[n]$ and let $\mathbb{X}_r\!:=\mathbb{R}^{n\times r}\times\mathbb{R}^{m\times r}$. We are interested in the following low-rank composite factorization model
\begin{equation}\label{prob}
  \min_{(U,V)\in\mathbb{X}_r}\Phi_{\lambda,\mu}(U,V):=f(UV^{\top}) +\lambda\sum_{i=1}^r\big[\theta(\|U_i\|)+\theta(\|V_i\|)\big]
  +\frac{\mu}{2}\big[\|U\|_F^2+\|V\|_F^2\big],
  \end{equation}
 where $f\!:\mathbb{R}^{n\times m}\to\mathbb{R}$ is a lower bounded $L_{\!f}$-smooth function (i.e., $f$ is differentiable and its gradient is Lipschitz continuous with modulus $L_{\!f}$), $\lambda>0$ and $\mu>0$ are the regularization parameters, and $\theta\!:\mathbb{R}\to\overline{\mathbb{R}}:=(-\infty,\infty]$ is a proper lower semicontinuous (lsc) function to promote sparsity and satisfy the following two conditions:

 \begin{description}
  \item[\bf(C.1)] $\theta(0)=0$, $\theta(t)>0$ for $t>0$, and $\theta(t)=\infty$ for $t<0$;

  \item[\bf(C.2)] $\theta$ is differentiable on $\mathbb{R}_{++}$, and its proximal mapping has a closed-form.
 \end{description}
 For convenience, in the sequel, we assume that $n\le m$, and $l=n$ or $m$, and write
 \begin{equation}\label{def-vtheta}
  F(U,V):=f(UV^{\top})\ \ {\rm for}\ (U,V)\in\mathbb{X}_{r}  \ \ {\rm and}\ \ \vartheta(W):={\textstyle\sum_{i=1}^r}\theta(\|W_i\|)\ \ {\rm for}\ W\!\in\mathbb{R}^{l\times r}.
\end{equation}

 The column sparsity regularization term $\lambda\sum_{i=1}^r[\theta(\|U_i\|)+\theta(\|V_i\|)]$ aims at reducing rank because the positive integer $r$ is usually a rough upper estimation for the true rank. Table \ref{tab1} below provides some common $\theta$ to satisfy conditions (C.1)-(C.2), where $\theta_4$ satisfies (C.2) in view of \cite{Cao13,Xu12}. When $\theta=\theta_1$, model \eqref{prob} becomes the column $\ell_{2,0}$-norm regularization problem studied in \cite{TaoQianPan22}; when $\theta=\theta_2$, it reduces to the factorization form of the nuclear-norm regularization problem (see \cite{Chen20,Hastie15}). The regularization term $\frac{\mu}{2}[\|U\|_F^2+\|V\|_F^2]$ plays a twofold role: one is to guarantee that \eqref{prob} has a nonempty set of optimal solutions and then a nonempty set of stationary points, and the other is to promote a balanced set of stationary points (see Proposition \ref{balance}).  Model \eqref{prob} has a wide application in matrix completion and sensing (see, e.g., \cite{Bhojanapalli16,Cai13,Jain13,SunLuo16}). 
 \begin{table}[h]
  \setlength{\abovecaptionskip}{-0cm}
  \setlength{\belowcaptionskip}{-0.01cm}
  \small
  \centering
  \caption{Some common functions satisfying conditions (C.1)-(C.2)}\label{tab1}
  \begin{tabular}{|c|c|}
  \hline
  $\theta_{1}(t)=\left\{\begin{array}{cl}
  {\rm sign}(t) &{\rm if}\ t\in\mathbb{R}_{+},\\
  \infty&{\rm otherwise}
 \end{array}\right.$ &	
  $\theta_{2}(t)=\left\{\begin{array}{cl}
  t^2 &{\rm if}\ t\in\mathbb{R}_{+},\\
  \infty&{\rm otherwise}
  \end{array}\right.$\\
  \hline	
  $\theta_{3}(t)=\left\{\begin{array}{cl}
  t &{\rm if}\ t\in\mathbb{R}_{+},\\
  \infty&{\rm otherwise}
 \end{array}\right.$ &
 $\theta_{4}(t)=\left\{\begin{array}{cl}
  t^{p} &{\rm if}\ t\in\mathbb{R}_{+},\\
  \infty&{\rm otherwise}
 \end{array}\right.\ {\rm for}\ p=1/2\ {\rm or}\ {2}/{3}$\\ 
 \hline
 \end{tabular}
 \end{table} 
 \subsection{Related work}\label{sec1.1}
 Problem \eqref{prob} is a special case of the composite optimization models considered in \cite{Attouch10,Bolte14,Chouzenoux16,Ochs19,Pock16,XuYin17}, where the nonsmooth regularization term has a separable structure and the nonsmooth function associated with each block has a closed-form proximal mapping. Hence, the proximal alternating linearized minimization (PALM) methods and their inertial versions developed in \cite{Attouch10,Bolte14,Pock16,XuYin17} are applicable to it. 
 The basic idea of these PALM methods is to minimize alternately the proximal version of the linearization of $\Phi_{\lambda,\mu}$ at the current iterate $(U^k,V^k)$. Their iterations depend on the Lipschitz modulus of the partial gradients $\nabla_{U}F(\cdot,V^k)$ and $\nabla_{V}F(U^k,\cdot)$, which are respectively $L_1(V^k)=L_{\! f}\|V^k\|^2$ and $L_1(U^k)=L_{\! f}\|U^{k}\|^2$. These constants are available once that of $\nabla\!f$ is known, but they are usually much larger than that of $\nabla\!f$; for example, when the spectral norm $\|U^{k}\|$ is greater than $10^3$ and $L_{\! f}\ge 1$, the Lipschitz modulus $L_1(V^k)$ is more than $10^6$. A similar case also occurs to $L_1(U^k)$. This brings a great difficulty to first-order methods for computing subproblems because their iteration complexity and local convergence rate become worse as $L_1(V^k)$ or $L_1(U^k)$ increases. Then, it is natural to ask whether an efficient alternating minimization (AM) algorithm can be designed by using the Lipschitz continuity of $\nabla\!f$ only. The block coordinate variable metric algorithms in \cite{Chouzenoux16,Ochs19} can be used to solve \eqref{prob}, but their efficiency depends on the appropriate choice of variable metric linear operators, and their convergence analysis requires the exact solutions of variable metric subproblems. Now it is unclear which kind of variable metric linear operators is suitable for problem \eqref{prob}, and whether the variable metric subproblems involving $\vartheta$ have a closed-form solution.

 Recently, for problem \eqref{prob} with $\theta=\theta_1$, the authors made an attempt to design an AM method by leveraging the Lipschitz continuity of $\nabla\!f$ only; see \cite[Algorithm 2]{TaoQianPan22}. They imposed a subspace correction strategy on every subproblem so that the corrected subproblems have a closed-form solution. Undoubtedly, this brings a great challenge to its convergence analysis, especially the full convergence analysis of the iterate sequence, though numerically it was 
 validated to be potentially superior to the PALM in terms of the quality of solutions and the running time via matrix completion problems with $f(X)=\frac{1}{2}\|P_\Omega(X\!-\!M)\|_F^2$. Until now, there are no convergence results for the generated iterate sequence. It is worth pointing out that such a subspace correction technique first appeared in softImpute-ALS, an alternating least squares method proposed in \cite{Hastie15} for the nuclear-norm regularized least squares factorization model, but as far as we know, there are no any convergence results obtained for it.
 \subsection{Main contribution}\label{sec1.2}

 The main contribution of this work is to conduct a systematic convergence analysis for the subspace correction AM method proposed in \cite{Hastie15,TaoQianPan22}, and provide a convergence certificate for its iterate sequence. Specifically, we prove the subsequence convergence of the iterate sequence when $\theta$ takes any of $\theta_i$ in Table \ref{tab1}, and establish the convergence of the whole iterate sequence and column subspace sequences of factor pairs under the KL property of $\Phi_{\lambda,\mu}$ and the restricted condition in \eqref{ass3-Uk}-\eqref{ass3-Vk}, which can be satisfied by $\theta_1$, or by $\theta_4$ if there is an  accumulation point with distinct nonzero singular values. To the best of our knowledge, this is the first subspace correction AM method with convergence certificate for low-rank composite factorization models. Hastie et al. \cite{Hastie15} proposed a proximal AM algorithm with subspace correction (named softImpute-ALS) for the factorization form of the nuclear-norm regularized least squares model, but did not provide the convergence analysis of the iterate sequence even the objective value sequence. In fact, our Theorem \ref{converge-obj} on the convergence of the objective value sequence and Theorem \ref{Sub-convergence} on the subsequence convergence are applicable to softImpute-ALS.

 As will be shown in Section \ref{sec3}, our proximal AM algorithm is actually a variable metric proximal AM method, but different from the variable metric methods in \cite{Chouzenoux16,Ochs19}, our variable metric linear operators are natural products of majorizing $F$. In particular, by introducing a subspace correction step to per subproblem, we overcome the difficulty that the variable metric proximal subproblems have no closed-form solutions.

  \subsection{Notation}\label{sec1.3}
 In this paper, $\mathbb{O}^{n_1\times n_2}$ represents the set of all $n_1\times n_2$ real matrices with orthonormal columns,  $\mathbb{O}^{n_1}$ stands for $\mathbb{O}^{n_1\times n_1}$, and $I_r$ denotes the $r\times r$ identity matrix. For an integer $k\ge 1$, write $[k]\!:=\{1,\ldots,k\}$. For a matrix $X\!\in\mathbb{R}^{ n\times m}$, $\|X\|,\,\|X\|_*$ and $\|X\|_{2,0}$ denote the spectral norm, nuclear norm, and column $\ell_{2,0}$-norm of $X$, respectively, $\mathbb{B}(X,\delta)$ denotes the closed ball centered at $X$ with radius $\delta$ on the Frobenius norm, $\sigma(X)\!:=(\sigma_1(X),\ldots,\sigma_n(X))^{\top}$ with $\sigma_1(X)\ge\cdots\ge\sigma_n(X)$, $\Sigma_{\kappa}(X)\!:={\rm Diag}(\sigma_1(x),\ldots,\sigma_{\kappa}(X))$ and $\mathbb{O}^{n,m}(X)\!:=\{(U,V)\in\mathbb{O}^{n}\times\mathbb{O}^{m}\ |\ X\!=\!U[{\rm Diag}(\sigma(X))\ \ 0]V^{\top}\}$. For any $t\in \mathbb{R}$, ${\rm sgn}(t)=1$ for $t\in \mathbb{R}_{+}$, otherwise ${\rm sgn}(t)=-1$. For a matrix $Z\in\mathbb{R}^{n_1\times n_2}$, $Z_j$ denotes the $j$th column of $Z$, $J_{Z}$ represents its index set of nonzero columns, and ${\rm col}(Z)$ and ${\rm row}(Z)$ denote the subspace spanned by all columns and rows of $Z$. For an index set $J\subset[r]$, let $\overline{J}:=[r]\backslash J$ and define
 \begin{equation}\label{def-vthetaJ}
 \vartheta_{\!J}(Z_J)={\textstyle\sum_{i\in J}}\,\theta(\|Z_i\|)\ \  {\rm and}\  \ \vartheta_{\overline{J}}(Z_{\overline{J}})={\textstyle\sum_{i\in \overline{J}}}\,\theta(\|Z_i\|)\ \ {\rm for}\
 Z\in\mathbb{R}^{l\times r}.
\end{equation}
\section{Preliminaries}\label{sec2}

Recall that for a proper lsc function $h\!:\mathbb{R}^n\to\overline{\mathbb{R}}$, its proximal mapping associated with parameter $\gamma>0$ is defined as
 \(
   \mathcal{P}_{\!\gamma}h(z)\!:=\!\mathop{\arg\min}_{x\in \mathbb{R}^n}
   \big\{\frac{1}{2\gamma}\|z\!-\!x\|^2+h(x)\big\}
 \)
 for $z\in\mathbb{R}^n$. The mapping $\mathcal{P}_{\!\gamma}h$ is generally multi-valued unless the function $h$ is convex. The following lemma presents the proximal mapping of the function $\vartheta$ in \eqref{def-vtheta} with the help of that of $\theta$. Since its proof is trivial by condition (C.1), we omit it.
 \begin{lemma}\label{lemma-proxmap}
 Fix any $\gamma>0$ and $u\in\mathbb{R}^n$. Let $S^*$ be the solution set of the problem
 \begin{equation*}
 \min_{x\in\mathbb{R}^n}\Big\{\frac{1}{2}\|\gamma x -u\|^2+\lambda\theta(\|x\|)\Big\}.
 \end{equation*}
 Then, $S^*=\{0\}$ when $u=0$; otherwise $S^*=\frac{u}{\|u\|}\mathcal{P}_{\!\lambda/\gamma^2}\theta(\|u\|/\gamma)$.
 \end{lemma}

\subsection{Stationary points of problem \eqref{prob}}\label{sec2.2}

 To introduce stationary points of \eqref{prob}, we recall from \cite{RW98} the (basic) subdifferential.
 \begin{definition}\label{gsubdiff}
  Consider a function $h\!:\mathbb{R}^{n\times m}\to\overline{\mathbb{R}}$
  and a point $x$ with $h(x)$ finite.
  The regular subdifferential of $h$ at $x$, denoted by
  $\widehat{\partial}h(x)$, is defined as
  \[
    \widehat{\partial}h(x):=\bigg\{v\in\mathbb{R}^{n\times m}\ \big|\
    \liminf_{x\ne x'\to x}
    \frac{h(x')-h(x)-\langle v,x'-x\rangle}{\|x'-x\|_F}\ge 0\bigg\},
  \]
  and the basic (known as limiting or Morduhovich) subdifferential of $h$ at $x$ is defined as
  \[
    \partial h(x):=\Big\{v\in\mathbb{R}^{n\times m}\ |\  \exists\,x^k\to x\ {\rm with}\ h(x^k)\to h(x)\ {\rm and}\ v^k\to v\ {\rm with}\
    v^k\in\widehat{\partial}h(x^k)\Big\}.
  \]
 \end{definition}

 The following lemma characterizes the subdifferential of the function $\vartheta$ in \eqref{def-vtheta}. 
\begin{lemma}\label{subdiff-vtheta}
 Let $\phi(x)\!:=\theta(\|x\|)$ for $x\in\mathbb{R}^l$. Consider any $(\overline{U},\overline{V})\in\mathbb{X}_r$. Then, it holds $\partial\vartheta(\overline{U})=\overline{\mathcal{U}}_1\times\cdots\times\overline{\mathcal{U}}_r$ and $\partial\vartheta(\overline{V})=\overline{\mathcal{V}}_1\times\cdots\times\overline{\mathcal{V}}_r$ with
 \[
  \mathcal{\overline{U}}_j\!=\!\left\{\begin{array}{cl}
  \Big\{\frac{\theta'(\|\overline{U}_j\|)\overline{U}_j}{\|\overline{U}_j\|}\Big\}&\ {\rm if}\ \overline{U}_j\!\ne\! 0,\\
  \partial\phi(0)&\ {\rm if}\ \overline{U}_{\!j}\!=\!0
  \end{array}\right.\ {\rm and}\
  \mathcal{\overline{V}}_j\!=\!\left\{\begin{array}{cl}
  \Big\{\frac{\theta'(\|\overline{V}_j\|)\overline{V}_j}{\|\overline{V}_j\|}\Big\}&\ {\rm if}\ \overline{V}_j\!\ne \!0,\\
  \partial\phi(0)&\ {\rm if}\ \overline{V}_{\!j}=0,
  \end{array}\right.
 \]
 When $\theta$ is any of the functions in Table \ref{tab1}, it holds $0\in\partial\phi(0)$. 
 \end{lemma}
 \begin{proof}
 Note that $\vartheta(U)=\sum_{j=1}^r\phi(U_j)$ for $U\in\mathbb{R}^{n\times r}$. Invoking \cite[Proposition 10.5]{RW98} leads to $\partial\vartheta(\overline{U})=\partial\phi(\overline{U}_1)\times\cdots\times\partial\phi(\overline{U}_r)$. Fix any $j\in[r]$. When $\overline{U}_j\ne 0$, using condition (C.2) leads to $\partial\phi(\overline{U}_j)=\big\{\theta'(\|\overline{U}_j\|)\frac{U_j}{\|U_j\|}\big\}$. When $\overline{U}_j=0$, by Definition \ref{gsubdiff}, it is not hard to calculate that $\widehat{\partial}\phi(0)=\mathbb{R}$ for $\theta=\theta_1$ and $\theta_4$, $\widehat{\partial}\phi(0)=\{0\}$ for $\theta=\theta_2$, and $\widehat{\partial}\phi(0)=\mathbb{B}$ for $\theta=\theta_3$, where $\mathbb{B}$ is the unit ball of $\mathbb{R}^{l}$ centered at the origin. Together with $\widehat{\partial}\phi(0)\subset\partial\phi(0)$, we have $0\in\partial\phi(0)$ for $\theta=\theta_1$-$\theta_4$. The proof is completed.  
 \end{proof}

 Now we are ready to introduce the notion of stationary points for problem \eqref{prob}. 
 \begin{definition}\label{def-spoint}
 A factor pair $(\overline{U},\overline{V})\in\mathbb{X}_r$ is called a stationary point of problem \eqref{prob} if
 \[
   0\in\partial\Phi_{\lambda,\mu}(\overline{U},\overline{V})
   =\begin{pmatrix}
    \nabla f(\overline{U}\overline{V}^{\top})\overline{V} +\mu \overline{U}+\lambda[\,\mathcal{\overline{U}}_1\times\cdots\times\mathcal{\overline{U}}_r]\\
    \big[\nabla f(\overline{U}\overline{V}^{\top})\big]^{\top}\overline{U}+\mu \overline{V}+\lambda[\,\mathcal{\overline{V}}_1\times\cdots\times\mathcal{\overline{V}}_r]
    \end{pmatrix},
  \]
  where, for each $j\in[r]$, the sets $\mathcal{\overline{U}}_j$ and $\mathcal{\overline{V}}_j$ take the same form as in Lemma \ref{subdiff-vtheta}.
 \end{definition}

\subsection{Relation between column $\ell_{2,p}$-norm and Schatten $p$-norm}\label{sec2.3}

 For any $p>0$, the column $\ell_{2,p}$-norm and Schatten $p$-norm of a matrix $X\in\mathbb{R}^{n\times m}$ are respectively defined as $\|X\|_{2,p}\!:=\big(\sum_{j=1}^m\|X_j\|^p\big)^{{1}/{p}}$ and $\|X\|_{\rm S_p}\!:=\big(\sum_{i=1}^n[\sigma_i(X)]^p\big)^{{1}/{p}}$.
 The following lemma states that $\|X\|_{\rm S_p}$ is not more than $\ell_{2,p}$-norm $\|X\|_{2,p}$ if $p\in(0,1]$.
\begin{lemma}\label{lemma-Sp}
 Fix any $X\in\mathbb{R}^{n\times m}$. For any $p\in (0,1]$, it holds $\|X\|^{p}_{\rm S_{p}}\le\|X\|^{p}_{2,p}$. 
 \end{lemma}
 \begin{proof}
  Fix any $p\in (0,1]$. Write $r={\rm rank}(X)$. We first claim that for any $R\in\mathbb{O}^{d\times r}$ with $r\le d$, 
  \begin{equation}\label{temp-spnorm}
  \|X\|^{p}_{\rm S_{p}}\le\textstyle{ \sum_{i=1}^d}\big[(R\Sigma_r(X) R^{\top})_{ii}\big]^p.
 \end{equation} 
 Fix any $R\in\mathbb{O}^{d\times r}$ with  $r\le d$. For each $i\in[d]$, an elementary calculation leads to $(R\Sigma_r(X)R^{\top})_{ii}=\sum_{j=1}^r R_{ij}^2\sigma_j(X)$, which implies that
 \[
  \sum_{i=1}^d\big[(R\Sigma_r(X) R^{\top})_{ii}\big]^p=\sum_{i=1}^d\Big(\sum_{j=1}^rR_{ij}^2\sigma_j(X)\Big)^p.
 \]
 For each $i\in[d]$, let $\alpha_i=\sum_{j=1}^rR_{ij}^2$. As $R\in\mathbb{O}^{d\times r}$, we have $\alpha_i\in[0,1]$ for each $i\in[d]$. Note that the function $\mathbb{R}_{+}\ni t\mapsto t^p$
 is concave due to $p\in(0,1]$. Moreover, if $\alpha_i\ne 0$, $\sum_{j=1}^r{R_{ij}^2}/{\alpha_i}=1$. For each $i\in[d]$ with $\alpha_i\ne 0$, from Jensen's inequality, it holds
 \[
  \Big(\sum\limits_{j=1}^r\frac{R_{ij}^2}{\alpha_i}\sigma_j(X)\Big)^p\ge \sum\limits_{j=1}^r\frac{R_{ij}^2}{\alpha_i}(\sigma_j(X))^p,
\]
which by $\alpha_i\!\in\!(0,1]$ and $0\!<\!p\!\leq\! 1$ implies that
$\big(\sum_{j=1}^rR_{ij}^2\sigma_j(X)\big)^p\!\ge\! \sum_{j=1}^rR_{ij}^2(\sigma_j(X))^p$.
Together with $1=\sum_{i=1}^dR_{ij}^2=\sum_{i:\,\alpha_i\ne 0}R_{ij}^2$ for each $j\in[r]$, it follows that
\begin{align*}
 \sum_{i=1}^d\big[(R\Sigma_r(X) R^{\top})_{ii}\big]^p
 &=\sum_{i=1}^d\Big(\sum_{j=1}^rR_{ij}^2\sigma_j(X)\Big)^p
 =\sum_{i:\,\alpha_i\ne 0}\Big(\sum_{j=1}^rR_{ij}^2\sigma_j(X)\Big)^p\nonumber\\
 &\ge\sum_{i:\,\alpha_i\ne 0}\sum_{j=1}^rR_{ij}^2(\sigma_j(X))^p = \sum_{j=1}^r(\sigma_j(X))^p.
 \end{align*}
 This shows that inequality \eqref{temp-spnorm} holds. Now let $X$ have the skinny SVD given by $X=P_1\Sigma_{r}(X) Q_1^{\top}$ with $P_1\in \mathbb{O}^{n\times r}$ and $Q_1\in \mathbb{O}^{m\times r}$. From the definition of $\|X\|^{p}_{\rm S_{p}}$, 
 \begin{equation*}
 \|X\|^{p}_{\rm S_{p}}\!=  \!\sum_{i=1}^r(\sigma_{i}^2(X))^{p/2}\stackrel{\eqref{temp-spnorm}}{\le}\sum_{i=1}^m\big[(Q_1\Sigma_r(X)^2Q_1^{\top})_{ii}\big]^{p/2}\!=\!\sum_{i=1}^m\big[(X^{\top}X)_{ii}\big]^{{p}/{2}}\!=\!\|X\|_{2,p}^p
 \end{equation*}
 where the third equality is due to  $\sum_{i=1}^m\big[(X^{\top}X)_{ii}\big]^{{p}/{2}}=\sum_{j=1}^m\|X_j\|^p=\|X\|_{2,p}^p$.
 \end{proof}

 By invoking Lemma \ref{lemma-Sp}, we obtain the factorization form of the Schatten $p$-norm.
\begin{proposition}\label{Factor-Sp}
 Fix any $p\in(0,{1}/{2}]$. For any $X\in\mathbb{R}^{n\times m}$ with ${\rm rank}(X)\le d$, we have
 \[
  2\|X\|_{\rm S_p}^p =\min_{U\in\mathbb{R}^{n\times d},V\in \mathbb{R}^{m\times d}}\Big\{\|U\|_{2,2p}^{2p}+\|V\|_{2,2p}^{2p}\ \ {\rm s.t.}\ \ X=UV^{\top}\Big\}.
\]
\end{proposition}
 \begin{proof}
 From \cite[Theorem 2]{Shang20}, it follows that the following relation holds
 \[
 \|X\|_{\rm S_p} =\min_{U\in \mathbb{R}^{n\times d},V\in \mathbb{R}^{m\times d}\atop X=UV^{\top}}\frac{1}{2^{1/p}}\Big(\|U\|^{2p}_{\rm S_{2p}}+\|V\|^{2p}_{\rm S_{2p}}\Big)^{1/p}.
 \]
 Together with $0<2p\leq 1$ and Lemma \ref{lemma-Sp}, we immediately obtain that
 \[
  \|X\|_{\rm S_p} \leq \min_{U\in \mathbb{R}^{n\times d},V\in \mathbb{R}^{m\times d}\atop X=UV^{\top}} \frac{1}{2^{1/p}}\Big(\|U\|^{2p}_{2,2p}+\|V\|^{2p}_{2,2p}\Big)^{1/p}.
 \]
 Let $X$ have the skinny SVD as $X=P_1\Sigma_d(X)Q_1^{\top}$ with
 $P_1\in \mathbb{O}^{n\times d}$ and $Q_1\in \mathbb{O}^{m\times d}$. Taking $\overline{U}=P_1[\Sigma_{d}(X)]^{1/2}$ and $\overline{V}=Q_1[\Sigma_d(X)]^{1/2}$ yields $\|\overline{U}\|^{2p}_{2,2p}=\|\overline{V}\|^{2p}_{2,2p}=\|X\|^p_{\rm S_p}$. Together with the above inequality, we obtain the desired result.
\end{proof}
 \subsection{Kurdyka-{\L}ojasiewicz property}\label{sec2.4}
 Now we recall from \cite{Attouch10} the KL property of an extended real-valued function.
\begin{definition}\label{KL-def}
 For every $\eta>0$, denote $\Upsilon_{\!\eta}$ by the set of continuous concave functions $\varphi\!:[0,\eta)\to\mathbb{R}_{+}$ that are continuously differentiable in $(0,\eta)$ with $\varphi(0)=0$ and $\varphi'(s)>0$ for all $s\in(0,\eta)$. A proper function $h\!:\mathbb{R}^n\to\overline{\mathbb{R}}$ is said to have the KL property at a point $\overline{x}\in{\rm dom}\,\partial h$ if there exist $\delta>0,\eta\in(0,\infty]$ and $\varphi\in\Upsilon_{\!\eta}$ such that 
 \[
  \varphi'(h(x)\!-\!h(\overline{x})){\rm dist}(0,\partial h(x))\ge 1
 \]
 for all $x\in\mathbb{B}(\overline{x},\delta)\cap[h(\overline{x})<h<h(\overline{x})+\eta]$. If $h$ satisfies the KL property at each point of ${\rm dom}\,\partial h$, it is called a KL function.
\end{definition}

 
\begin{remark}\label{KL-remark}
 By Lemma 2.1 of \cite{Attouch10}, a proper lsc function has the KL property at every noncritical point. Thus, to prove that a proper lsc function $h\!:\mathbb{X}\to\overline{\mathbb{R}}$ is a KL function, it suffices to check whether $h$ has the KL property at any critical point.
\end{remark}

\section{A majorized PAM with subspace correction}\label{sec3}

 Fix any $(U',V')\in\mathbb{X}_r$. Since $\nabla\!f$ is Lipschitz continuous with modulus $L_{\!f}$, it holds 
 \begin{align*}
  f(UV^{\top})
  &\le \widehat{F}(U,V,U',V'):=f(U'{V'}^{\top})+\langle\nabla\!f(U'{V'}^{\top}),UV^{\top}\!-\!U'{V'}^{\top}\rangle\\
  &\quad+\frac{L_{\!f}}{2}\|UV^{\top}\!-\!U'{V'}^{\top}\|_F^2\quad{\rm for\ any}\ (U,V)\in\mathbb{X}_r,
 \end{align*}
 which together with the expression of $\Phi_{\lambda,\mu}$ immediately implies that
 \begin{equation*}
  \Phi_{\lambda,\mu}(U,V)
 \le\widehat{\Phi}_{\lambda,\mu}(U,V,U',V')\!:=\!\widehat{F}(U,V,U',V')\!+\!\lambda\big[\vartheta(U)\!+\!\vartheta(V)\big]\!+\frac{\mu}{2}\big(\|U\|_F^2\!+\!\|V\|_F^2\big).
 \end{equation*}
 Note that $\widehat{\Phi}_{\lambda,\mu}(U',V',U',V')\!=\!\Phi_{\lambda,\mu}(U',V')$, so   $\widehat{\Phi}_{\lambda,\mu}(\cdot,\cdot,U',V')$ is a majorization of $\Phi_{\lambda,\mu}$ at $(U',V')$.
 Let $(U^k,V^k)$ be the current iterate. It is natural to develop an algorithm for solving \eqref{prob} by minimizing the function $\widehat{\Phi}_{\lambda,\mu}(\cdot,\cdot,U^k,V^k)$ alternately:
 \begin{subnumcases}{}
 U^{k+1}\in\mathop{\arg\min}_{U\in\mathbb{R}^{n\times r}}\Big\{\langle\nabla_{U}\!F(U^k,V^k),U\rangle\!+\!\lambda\vartheta(U)\!+\!\frac{\mu}{2}\|U\|_F^2+\frac{L_{\!f}}{2}\|(U-U^k)(V^k)^{\top}\|_F^2\Big\} \nonumber,\\
V^{k+1}\!\in\!\mathop{\arg\min}_{V\in\mathbb{R}^{m\times r}}\Big\{\langle\nabla_{V}\!F(U^{k+1},V^k),V\rangle\!+\!\lambda\vartheta(V)\!+\!\frac{\mu}{2}\|V\|_F^2+\frac{L_{\!f}}{2}\|U^{k+1}(V\!-\!V^k)^{\top}\|_F^2\Big\}.\nonumber
 \end{subnumcases}
 Compared with the PALM method, such a majorized AM method is actually a variable metric proximal AM method. Unfortunately, due to the nonsmooth regularizer $\vartheta$, the variable metric subproblems have no closed-form solutions, which brings a great challenge for convergence analysis of the generated iterate sequence. Inspired by the fact that variable metric proximal methods are more effective especially for ill-conditioned problems, we introduce a subspace correction step to per proximal subproblem to guarantee that the variable metric proximal subproblem at the corrected factor has a closed-form solution, and propose a majorized proximal AM algorithm (PAM) with subspace correction. Its iteration steps are described as follows.
 \begin{algorithm}[h]
  \caption{\label{MPAMSC}{\bf (A majorized PAM with subspace correction)}}
  \textbf{Initialization:} Input parameters $\varrho\in(0,1),\underline{\gamma_1}>0,
  \underline{\gamma_2}>0,\gamma_{1,0}>0$ and $\gamma_{2,0}>0$.
  Choose $\overline{P}^0\!\in\mathbb{O}^{m\times r},\widehat{P}^0\!\in\mathbb{O}^{n\times r},\overline{Q}^0=\overline{D}^0=I_{r}$, and let
  $(\overline{U}^0,\overline{V}^0)=(\widehat{P}^0,\overline{P}^0)$. \\
 \textbf{For} $k=0,1,2,\ldots$ \textbf{do}
 \begin{itemize}
  \item[1.] Compute
            \(                  U^{k+1}\in\displaystyle{\mathop{\arg\min}_{U\in\mathbb{R}^{n\times r}}}
                 \Big\{\widehat{\Phi}_{\lambda,\mu}(U,\overline{V}^{k},\overline{U}^k,\overline{V}^k)          +\frac{\gamma_{1,k}}{2}\|U-\overline{U}^k\|_F^2\Big\}.
           \)

   \item[2.] Perform a thin SVD for $U^{k+1}\overline{D}^k$ such that $U^{k+1}\overline{D}^k=\widehat{P}^{k+1}(\widehat{D}^{k+1})^2(\widehat{Q}^{k+1})^{\top}$ with $\widehat{P}^{k+1}\in\mathbb{O}^{n\times r}, \widehat{Q}^{k+1}\in\mathbb{O}^{r}$ and $(\widehat{D}^{k+1})^2={\rm Diag}(\sigma(U^{k+1}\overline{D}^k))$, and set
    \[
      \widehat{U}^{k+1}:=\widehat{P}^{k+1}\widehat{D}^{k+1},\ \
      \widehat{V}^{k+1}\!:=\overline{P}^k\widehat{Q}^{k+1}\widehat{D}^{k+1} \ \ {\rm and}\  \ \widehat{X}^{k+1}\!:=\widehat{U}^{k+1}(\widehat{V}^{k+1})^{\top}.
    \]

 \item[3.] Compute
           \(
             V^{k+1}\in\displaystyle{\mathop{\arg\min}_{V\in \mathbb{R}^{m\times r}}}
             \Big\{\widehat{\Phi}_{\lambda,\mu}(\widehat{U}^{k+1},V,\widehat{U}^{k+1},\widehat{V}^{k+1})
           +\frac{\gamma_{2,k}}{2}\|V-\widehat{V}^{k+1}\|_F^2\Big\}.
          \)

 \item[4.] Find a thin SVD for $V^{k+1}\widehat{D}^{k+1}$ such that               $V^{k+1}\widehat{D}^{k+1}=\overline{P}^{k+1}(\overline{D}^{k+1})^2(\overline{Q}^{k+1})^{\top}$ with $\overline{P}^{k+1}\in\mathbb{O}^{m\times r},\overline{Q}^{k+1}\in\mathbb{O}^{r}$ and $(\overline{D}^{k+1})^2={\rm Diag}(\sigma(V^{k+1}\widehat{D}^{k+1}))$, and set
 \[                 \overline{U}^{k+1}\!:=\widehat{P}^{k+1}\overline{Q}^{k+1}\overline{D}^{k+1},\ \
        \overline{V}^{k+1}\!:=\overline{P}^{k+1}\overline{D}^{k+1}\ \ {\rm and}\  \ \overline{X}^{k+1}\!:=\overline{U}^{k+1}(\overline{V}^{k+1})^{\top}\!.
       \]

 \item[5.] Set $\gamma_{1,k+1}=\max(\underline{\gamma_1},\varrho \gamma_{1,k})$ and             $\gamma_{2,k+1}=\max(\underline{\gamma_2},\varrho \gamma_{2,k})$.
 \end{itemize}
 \textbf{end (For)}
 \end{algorithm}
\begin{remark}\label{remark-MPAMSC}
 {\bf(a)} Steps 2 and 4 are the subspace correction steps. Step 2 is constructing a new factor pair $(\widehat{U}^{k+1},\widehat{V}^{k+1})$ by performing an SVD for $U^{k+1}\overline{D}^k$, whose column space ${\rm col}(\widehat{U}^{k+1})$ can be regarded as a corrected one for that of $U^{k+1}$. This step ensures that the proximal minimization of the majorization $\widehat{\Phi}_{\lambda,\mu}(\cdot,\cdot,\widehat{U}^{k+1},\widehat{V}^{k+1})$ with respect to $V$ has a closed-form solution. Similarly, step 4 is constructing a new factor pair $(\overline{U}^{k+1},\overline{V}^{k+1})$ by performing an SVD for $V^{k+1}\widehat{D}^{k+1}$, whose column space ${\rm col}(\overline{V}^{k+1})$ can be viewed as a corrected one for that of $\widehat{V}^{k+1}$. This step ensures that the proximal minimization of the majorization $\widehat{\Phi}_{\lambda,\mu}(\cdot,\cdot,\overline{U}^{k+1},\overline{V}^{k+1})$ with respect to $U$ has a closed-form solution. As will be shown in Theorem \ref{converge-obj} later, the objective value at  $(\widehat{U}^{k+1}\!,\widehat{V}^{k+1})$ is not more than the one at $(\overline{U}^{k},\overline{V}^k)$, while the objective value at  $(\overline{U}^{k+1},\overline{V}^{k+1})$ is not more than the one at $(\widehat{U}^{k+1},\widehat{V}^{k+1})$. This shows that the subspace correction steps contribute to reducing the objective values. From steps 2 and 4, for each $k\in\mathbb{N}$,
 \begin{subnumcases}{}{}\label{relation1-Uhat}     U^{k+1}\overline{D}^k(\overline{P}^k)^{\top}=U^{k+1}(\overline{V}^k)^{\top}=\widehat{X}^{k+1}=\widehat{U}^{k+1}(\widehat{V}^{k+1})^{\top},\\
  \label{relation1-Ubar}
  \widehat{P}^{k+1}\widehat{D}^{k+1}(V^{k+1})^{\top}=\widehat{U}^{k+1}(V^{k+1})^{\top}=\overline{X}^{k+1}=\overline{U}^{k+1}(\overline{V}^{k+1})^{\top}.
 \end{subnumcases}

 \noindent
 {\bf(b)} In steps 1 and 3, we introduce a proximal term to guarantee the sufficient decrease of the objective value sequence. The uniformly lower bound $\underline{\gamma_1}$ or $\underline{\gamma_2}$ for the proximal parameter is easy to choose, say $10^{-8}$, in practical computation. According to the optimality of $U^{k+1}$ and $V^{k+1}$ and Exercise 8.8 of \cite{RW98}, for each $k$, it holds 
 \begin{align*}
  0\in\big[\nabla\!f(\overline{X}^{k})+L_{\!f}(\widehat{X}^{k+1}\!-\!\overline{X}^{k})\big]\overline{V}^{k}+\mu U^{k+1}+\gamma_{1,k}(U^{k+1}\!-\overline{U}^k)+\lambda \partial\vartheta(U^{k+1});\qquad\\
  0\in\big[\nabla f(\widehat{X}^{k+1})\!+L_{\!f}(\overline{X}^{k+1}\!\!\!-\!\widehat{X}^{k+1})\big]^{\top}\widehat{U}^{k+1}\!+\!\mu {V}^{k+1}\!+\!\gamma_{2,k}({V}^{k+1}\!-\!\widehat{V}^{k+1})\!+\!\lambda\partial \vartheta(V^{k+1}).
 \end{align*}

 \noindent
 {\bf (c)} From step 1, equations \eqref{relation1-Uhat}-\eqref{relation1-Ubar} and the expression of $\widehat{\Phi}_{\lambda,\mu}(\cdot,\overline{V}^k,\overline{U}^k,\overline{V}^k\!)$, we get
 \[
   U^{k+1}\in\mathop{\arg\min}_{U\in\mathbb{R}^{n\times r}}
   \Big\{\frac{1}{2}\big\|G^k-U\Lambda^k\big\|_F^2+\lambda\sum_{i=1}^r\theta(\|U_i\|)\Big\}
 \]
 where $G^k\!:=\!\big(L_{\!f}\overline{Z}^k\overline{P}^k\!+\!\gamma_{1,k}\widehat{P}^k\overline{Q}^k\big)\overline{D}^k(\Lambda^k)^{-1}$
 with $\overline{Z}^k=\overline{X}^k\!-\!L_{\!f}^{-1}\nabla\!f(\overline{X}^k)$ and $\Lambda^k=\!\big[L_{\!f}(\overline{D}^k)^2\!+(\mu+\gamma_{1,k})I_r\big]^{1/2}$ for each $k$. Invoking Lemma \ref{lemma-proxmap} leads to 
 \begin{equation}\label{Uk-def}
  U_i^{k+1}\in\left\{\!\begin{array}{cl}
  \frac{G_i^k}{\|G_i^k\|}\mathcal{P}_{\!{\lambda}/{(\Lambda_{ii}^k)^2}}\theta\Big(\frac{\|G^k_i\|}{\Lambda_{ii}^k} \Big)& {\rm if}\ \|G_i^k\|>0,\\
  \{0\} &{\rm if}\ \|G_i^k\|=0
  \end{array}\right.\quad{\rm for\ each}\ i\in[r].
 \end{equation}
 While from step 3, equations \eqref{relation1-Uhat}-\eqref{relation1-Ubar} and the expression of $\widehat{\Phi}_{\lambda,\mu}(\widehat{U}^{k\!+\!1},\cdot,\widehat{U}^{k\!+\!1},\widehat{V}^{k\!+\!1})$,
 \[
   V^{k+1}\in\mathop{\arg\min}_{V\in\mathbb{R}^{m\times r}}
   \Big\{\frac{1}{2}\big\|H^{k+1}\!-\!V\Delta^{k+1}\big\|_F^2\!+\!\lambda\sum_{i=1}^r\theta(\|V_i\|)\Big\},
  \]
  where, for each $k\in\mathbb{N}$, $H^{k+1}=\big(L_{\!f}(\widehat{Z}^{k+1})^{\top}\widehat{P}^{k+1}+\gamma_{2,k}\overline{P}^k\widehat{Q}^{k+1}\big)
  \widehat{D}^{k+1}(\Delta^{k+1})^{-1}$ with
  $\widehat{Z}^{k+1}=\widehat{X}^{k+1}-L_{\!f}^{-1}\nabla\!f(\widehat{X}^{k+1})$ and $\Delta^{k+1}=\big[L_{\!f}(\widehat{D}^{k+1})^2\!+\!(\mu\!+\!\gamma_{2,k})I_r\big]^{1/2}$. By Lemma \ref{lemma-proxmap}, 
\begin{equation}\label{Vk-def}
   V_i^{k+1}\in\left\{\!\begin{array}{cl}
  \frac{H_i^{k+1}}{\|H_i^{k+1}\|}\mathcal{P}_{\!{\lambda}/{(\Delta_{ii}^{k+1})^2}}\theta\Big(\frac{\|H_i^{k+1}\|}{\Delta_{ii}^{k+1}}\Big)& {\rm if}\ \|H_i^{k+1}\|>0,\\
   \{0\} &{\rm if}\ \|H_i^{k+1}\|=0
 \end{array}\right.\quad{\rm for\ each}\ i\in[r].
 \end{equation}
 Recall that $\theta$ is assumed to have a closed-form proximal mapping, so the main cost of Algorithm \ref{MPAMSC} in each step is to perform an SVD for $U^{k+1}\overline{D}^k$ and $V^{k+1}\widehat{D}^{k+1}$, whose computation cost is cheap because $r$ is usually chosen to be far less than $\min\{m,n\}$.
\end{remark}

 From Remark \ref{remark-MPAMSC}, we have that Algorithm \ref{MPAMSC} is well defined. For its iterate sequences
 $\{(U^k,V^k)\}_{k\in\mathbb{N}}$, $\{(\widehat{U}^k,\widehat{V}^k)\}_{k\in\mathbb{N}}$ and
 $\{(\overline{U}^k,\overline{V}^k)\}_{k\in\mathbb{N}}$, the following propositions establish the relation among their column spaces and nonzero column indices. Since the proof of Proposition \ref{prop-nzind} is similar to that of \cite[Proposition 4.2 (iii)]{TaoQianPan22}, we here omit it.
 \begin{proposition}\label{prop-colspace}
  Let $\big\{(U^k,V^k,\widehat{U}^{k},\widehat{V}^{k},\overline{U}^{k},\overline{V}^{k},\widehat{X}^{k},\overline{X}^{k})\big\}_{k\in\mathbb{N}}$
  be the sequence generated by Algorithm \ref{MPAMSC}. Then, for every $k\in \mathbb{N}$, the following inclusions hold
  \begin{equation*}
  {\rm col}(\overline{U}^{k+1})\subset{\rm col}(\widehat{U}^{k+1})\subset{\rm col}(U^{k+1})\ \ {\rm and}\ \
  {\rm col}(\widehat{V}^{k+2})\subset{\rm col}(\overline{V}^{k+1})\subset{\rm col}(V^{k+1}).
 \end{equation*}
 \end{proposition}
 \begin{proof}
 From equation \eqref{relation1-Uhat}, ${\rm col}(\widehat{X}^{k+1})\subset{\rm col}({U}^{k+1})$ and ${\rm row}(\widehat{X}^{k+1})\subset{\rm col}(\overline{V}^{k})$. While from step 2 of Algorithm \ref{MPAMSC}, $\widehat{X}^{k+1}=\widehat{P}^{k+1}(\widehat{D}^{k+1})^2(\overline{P}^{k}\!\widehat{Q}^{k+1})^{\top}$, by which it is easy to check that ${\rm col}(\widehat{U}^{k+1})={\rm col}(\widehat{X}^{k+1})$ and ${\rm col}(\widehat{V}^{k+1})={\rm row}(\widehat{X}^{k+1})$. Then,
 \begin{equation}\label{colUk1}
  {\rm col}(\widehat{U}^{k+1})\subset{\rm col}(U^{k+1})\ \ {\rm and}\ \  {\rm col}(\widehat{V}^{k+1})\subset{\rm col}(\overline{V}^{k}).
 \end{equation}
 Similarly, from equation \eqref{relation1-Ubar}, ${\rm col}(\overline{X}^{k+1})\subset{\rm col}(\widehat{U}^{k+1})$ and ${\rm row}(\overline{X}^{k+1})\subset {\rm col}({V}^{k+1})$. From step 4 of Algorithm \ref{MPAMSC}, $\overline{X}^{k+1}=\widehat{P}^{k+1}\overline{Q}^{k+1}(\overline{D}^{k+1})^2(\overline{P}^{k+1})^{\top}$. Then, it holds that
 \begin{equation}\label{colUk2}
  {\rm col}(\overline{U}^{k+1})={\rm col}(\overline{X}^{k+1})\!\subset{\rm col}(\widehat{U}^{k+1}) \ \ {\rm and}\ \  {\rm col}(\overline{V}^{k+1})=\!{\rm row}(\overline{X}^{k+1})\subset{\rm col}({V}^{k+1}).
 \end{equation}
  From the above equations \eqref{colUk1}-\eqref{colUk2}, we immediately obtain the desired result.
 \end{proof}
 \begin{proposition}\label{prop-nzind}
  Let $\big\{(U^k\!,V^k\!,\widehat{U}^{k}\!,\widehat{V}^{k}\!,\overline{U}^{k}\!,\overline{V}^{k}\!,\widehat{X}^{k}\!,\overline{X}^{k})\big\}_{k\in\mathbb{N}}$
  be the sequence generated by Algorithm \ref{MPAMSC}. Then, there exists $\overline{k}\in\mathbb{N}$ such that for all $k\ge\overline{k}$,
  \begin{align}\label{index}
  J_{V^{k}}=J_{U^{k}}=J_{\widehat{U}^{k}}=J_{\widehat{V}^{k}}  =J_{\overline{V}^{k}}\!=\!J_{\overline{U}^{k}}\!=\!J_{\overline{U}^{k+1}},\qquad\qquad\\
  \label{rank-equal}
  \!\!{\rm rank}(\overline{X}^{k})={\rm rank}(\widehat{X}^{k})\!=\!
 {\rm rank}(\widehat{U}^{k})\!=\!{\rm rank}(\widehat{V}^{k})
 \!=\!{\rm rank}(\overline{U}^{k})\!=\!{\rm rank}(\overline{V}^{k})=\!\|\overline{U}^k\|_{2,0},
 \end{align}
 and hence ${\rm col}(\overline{U}^{k+1})\!=\!{\rm col}(\widehat{U}^{k+1})\!=\!{\rm col}(U^{k+1})$ and ${\rm col}(\widehat{V}^{k+1})\!=\!{\rm col}(\overline{V}^{k+1})\!=\!{\rm col}(V^{k+1})$.
 \end{proposition}
\section{Convergence analysis of Algorithm \ref{MPAMSC}}\label{sec4}

 To proceed the convergence analysis of Algorithm \ref{MPAMSC}, we need the following assumption. 
\begin{assumption}\label{ass1}
 For any given $X\!\in\mathbb{R}^{n\times m}$ of ${\rm rank}(X)\!\le\!\kappa$ and $(P,Q)\!\in\!\mathbb{O}^{n,m}(X)$, the factor pair $(\overline{U},\overline{V})=\!(P_1[\Sigma_{\kappa}(X)]^{\frac{1}{2}},Q_1[\Sigma_{\kappa}(X)]^{\frac{1}{2}})$ satisfies
 \begin{equation*}
  \vartheta(\overline{U})+\vartheta(\overline{V})
  =\inf_{(U,V)\in\mathbb{R}^{n\times\kappa}\times\mathbb{R}^{m\times \kappa}}\Big\{\vartheta(U)+\vartheta(V)\ \ {\rm s.t.}\ \ X=UV^{\top}\Big\},
 \end{equation*}
 where $P_1$ and $Q_1$ are the submatrix consisting of the first $\kappa$ columns of $P$ and $Q$.
 \end{assumption}

 Assumption \ref{ass1} is rather mild, and it can be satisfied by the function $\vartheta$ associated with some common $\theta$ to promote sparsity. For example, when $\theta=\theta_1$, Assumption \ref{ass1} holds by \cite{BiTaoPan22}; when $\theta=\theta_2$, it holds due to \cite[Lemma 1]{Srebro05}; and when $\theta=\theta_3$-$\theta_4$, it holds by Proposition \ref{Factor-Sp}. Under Assumption \ref{ass1}, by following the similar proof to that of \cite[Proposition 4.2 (i)]{TaoQianPan22}, we can establish the convergence of the objective value sequences.
\begin{theorem}\label{converge-obj}
 Let $\big\{(U^k,V^k,\widehat{U}^{k},\widehat{V}^{k},\overline{U}^{k},\overline{V}^{k},\widehat{X}^{k},\overline{X}^{k})\big\}_{k\in\mathbb{N}}$ be the sequence generated by Algorithm \ref{MPAMSC}, and write $\underline{\gamma}:=\min\{\underline{\gamma_1},\underline{\gamma_2}\}$. Then, under Assumption \ref{ass1}, for each $k\in\mathbb{N}$,
 \begin{align}\label{descrease-ineq0}
 \Phi_{\lambda,\mu}(\overline{U}^{k},\overline{V}^{k})
 &\ge {\Phi}_{\lambda,\mu}(\widehat{U}^{k+1},\widehat{V}^{k+1})
     +(\underline{\gamma}/2)\|U^{k+1}-\overline{U}^k\|_F^2\\
 \label{descrease-ineq1}
 &\ge\Phi_{\lambda,\mu}(\overline{U}^{k+1},\overline{V}^{k+1})
    +\frac{\underline{\gamma}}{2}\big(\|U^{k+1}\!-\overline{U}^k\|_F^2
    +\|V^{k+1}\!-\widehat{V}^{k+1}\|_F^2\big),
 \end{align}
 so $\{\Phi_{\lambda,\mu}(\overline{U}^k,\overline{V}^k)\}_{k\in\mathbb{N}}$ and $\{\Phi_{\lambda,\mu}(\widehat{U}^k,\widehat{V}^k)\}_{k\in\mathbb{N}}$ converge to the same point, say, $\varpi^*$.
\end{theorem}

 As a direct consequence of Theorem \ref{converge-obj}, 
 we obtain the following corollary.
 \begin{corollary}\label{corollary4.1}
  Let $\big\{(U^k\!,V^k\!,\widehat{U}^{k}\!,\widehat{V}^{k}\!,\overline{U}^{k}\!,\overline{V}^{k}\!,\widehat{X}^{k}\!,\overline{X}^{k})\big\}_{k\in\mathbb{N}}$ be the sequence generated by Algorithm \ref{MPAMSC}. Then, under Assumption \ref{ass1}, the following assertions are true.
  \begin{description}
  \item[(i)] $\lim_{k\to\infty}\|U^{k+1}-\overline{U}^k\|_F=0$ and $\lim_{k\to\infty}\|V^{k+1}-\widehat{V}^{k+1}\|_F=0$;

  \item[(ii)] the sequence $\big\{(U^k,V^k,\widehat{U}^{k},\widehat{V}^{k},\overline{U}^{k},\overline{V}^{k},\widehat{X}^{k},\overline{X}^{k})\big\}_{k\in\mathbb{N}}$ is bounded;

  \item[(iii)] by letting $\overline{\beta}:=\sup_{k\in\mathbb{N}}\max\big\{\|\overline{V}^k\|,\|\widehat{U}^k\|,\|{V}^k\|,\|{U}^k\|\big\}$, for each $k\in\mathbb{N}$,
  \begin{align*}
  \Phi_{\lambda,\mu}(\overline{U}^{k},\overline{V}^{k})
   &\ge{\Phi}_{\lambda,\mu}(\overline{U}^{k+1},\overline{V}^{k+1})
      \! +\!\frac{\underline{\gamma}}{4}\big(\|U^{k+1}-\overline{U}^k\|_F^2
       \!+\!\|V^{k+1}\!-\widehat{V}^{k+1}\|_F^2\big)\nonumber\\
   &\quad\!+\!\frac{\underline{\gamma}}{8\overline{\beta}^2}\big(\|\widehat{X}^{k+1}\!-\!\overline{X}^k\|^2_F
     \! + \!\|\overline{X}^{k+1}\!\!-\!\widehat{X}^{k+1}\|^2_F\big)\!+\!\frac{\underline{\gamma}}{16\overline{\beta}^2} \|\overline{X}^{k}\!\!\!-\!\overline{X}^{k+1}\|^2_F;
   \end{align*}

 \item[(iv)] $\lim_{k\rightarrow \infty}\|\widehat{X}^{k+1}-\overline{X}^k\|_F=0$ and $\lim_{k\rightarrow \infty}\|\overline{X}^{k+1}-\widehat{X}^{k+1}\|_F=0$.
 \end{description}
 \end{corollary}
 \begin{proof}
 {\bf (i)-(ii)} Part (i) is immediate by Theorem \ref{converge-obj}, and we only need to prove part (ii).
 By Theorem \ref{converge-obj}, ${\Phi}_{\lambda,\mu}(\overline{U}^{k+1},\overline{V}^{k+1})\le{\Phi}_{\lambda,\mu}(\widehat{U}^{k+1},\widehat{V}^{k+1})\le{\Phi}_{\lambda,\mu}(\overline{U}^{0},\overline{V}^{0})$ for each $k\in\mathbb{N}$. Recall that $f$ is lower bounded, so the function   ${\Phi}_{\lambda,\mu}$ is coercive. Thus, the sequence $\big\{(\widehat{U}^{k},\widehat{V}^{k},               \overline{U}^{k},\overline{V}^{k})\big\}_{k\in\mathbb{N}}$ is bounded. Along with part (i), the sequence $\{(U^k,V^k)\}_{k\in\mathbb{N}}$ is bounded. In addition, the boundedness of
 $\big\{(\widehat{U}^{k},\widehat{V}^{k},             \overline{U}^{k},\overline{V}^{k})\big\}_{k\in\mathbb{N}}$ implies that of $\{(\widehat{X}^{k},\overline{X}^k)\}_{k\in\mathbb{N}}$ because $\widehat{X}^{k}=\widehat{U}^{k}(\widehat{V}^{k})^{\top}$ and $\overline{X}^{k}=\overline{U}^{k}(\overline{V}^{k})^{\top}$ for each $k$ by \eqref{relation1-Uhat}-\eqref{relation1-Ubar}.

 \noindent
 {\bf(iii)} Fix any $k\in\mathbb{N}$. From equations \eqref{relation1-Uhat}-\eqref{relation1-Ubar}, it follows that
 \begin{align*}
  \!\|\widehat{X}^{k+1}-\overline{X}^k\|_F=\big\|U^{k+1}(\overline{V}^k)^{\top}-\overline{U}^k(\overline{V}^k)^{\top}\big\|_F
 \le\|\overline{V}^k\|\|U^{k+1}-\overline{U}^{k}\|_F,\qquad\\
 \!\|\overline{X}^{k+1}\!-\widehat{X}^{k+1}\|_F
  =\big\|\widehat{U}^{k+1}({V}^{k+1})^{\top}\!-\widehat{U}^{k+1}(\widehat{V}^{k+1})^{\top}\big\|_F
  \le\|\widehat{U}^{k+1}\|\|V^{k+1}-\widehat{V}^{k+1}\|_F.
 \end{align*}
 Combining these two inequalities with the definition of $\overline{\beta}$ and Theorem \ref{converge-obj} leads to
 \begin{align*}
  \Phi_{\lambda,\mu}(\overline{U}^{k},\overline{V}^{k})
   &\ge{\Phi}_{\lambda,\mu}(\overline{U}^{k+1},\overline{V}^{k+1})
  +({\underline{\gamma}}/{4})\big(\|U^{k+1}-\overline{U}^k\|_F^2
  +\|V^{k+1}-\widehat{V}^{k+1}\|_F^2\big)\nonumber\\
  &\quad +\frac{\underline{\gamma}}{4\overline{\beta}^{2}}\big(\|\widehat{X}^{k+1}-\overline{X}^k\|^2_F
  +\|\overline{X}^{k+1}-\widehat{X}^{k+1}\|^2_F\big).
 \end{align*}
 Along with $2\|\widehat{X}^{k+1}-\overline{X}^k\|_F^2+2\|\overline{X}^{k+1}\!-\!\widehat{X}^{k+1}\|_F^2\ge\|\overline{X}^{k}\!-\!\overline{X}^{k+1}\|_F^2$, we get the result.

 \noindent
 {\bf(iv)} The result follows by part (iii) and the convergence of $\{\Phi_{\lambda,\mu}(\overline{U}^k,\overline{V}^k)\}_{k\in\mathbb{N}}$.
 \end{proof}
  \begin{remark}\label{remark41}
 By Corollary \ref{corollary4.1} (ii) and Remark \ref{remark-MPAMSC} (c), there exists a constant $\widehat{\beta}>0$ such that for all $k\in\mathbb{N}$ and $i\in[r]$,   $(\mu+\underline{\gamma})\le(\Lambda_{ii}^k)^2\le\widehat{\beta}$ and $(\mu+\underline{\gamma})\le(\Delta_{ii}^k)^2\le\widehat{\beta}$, where $\Lambda^k$ and $\Delta^k$ are the diagonal matrices appearing in Remark \ref{remark-MPAMSC} (c).
 \end{remark}
 \subsection{Subsequence convergence}\label{sec4.2}

 For each $k\in\mathbb{N}$, write $W^{k}\!:=\big(\widehat{U}^{k},\widehat{V}^{k},  \overline{U}^{k},\overline{V}^{k},\widehat{X}^{k},\overline{X}^{k}\big)$. To establish the subsequence convergence of $\{W^k\}_{k\in\mathbb{N}}$, that is, to prove that every accumulation point of $\{W^k\}_{k\in\mathbb{N}}$ is a stationary point of problem \eqref{prob}, we need the following technical lemma.
 \begin{lemma}\label{Lemma4.1}
  Under Assumption \ref{ass1}, the following assertions hold. 
  \begin{description}
  \item[(i)] The set of cluster points of $\{W^{k}\}_{k\in\mathbb{N}}$, denoted by $\mathcal{W}^*$, is nonempty and compact.		
		
  \item[(ii)] For each $W=(\widehat{U},\widehat{V},\overline{U},\overline{V},\widehat{X},\overline{X})\in\mathcal{W}^*$, it holds that
 $\widehat{U}\widehat{V}^{\top}=\widehat{X}=\overline{X}=\overline{U}\overline{V}^{\top}$ with ${\rm rank}(\overline{X})={\rm rank}(\widehat{X})=\|\overline{U}\|_{2,0}=\|\widehat{U}\|_{2,0}:=\overline{r}$,
  $(\widehat{U},\widehat{V})\!=\big(\widehat{P}[\Sigma_r(\widehat{X})]^{\frac{1}{2}},
  \widehat{R}[\Sigma_r(\widehat{X})]^{\frac{1}{2}}\big)$ for some $(\widehat{P},\widehat{R})\in\mathbb{O}^{n\times r}\times\mathbb{O}^{m\times r}$ such that $\widehat{X}\!\!=\!\!\widehat{P}\Sigma_{r}(\widehat{X})\widehat{R}^{\top}$, and $(\overline{U}\!,\overline{V})\!\!=\!\!\big(\overline{R}[\Sigma_r(\overline{X})]^{\frac{1}{2}},
  \overline{Q}[\Sigma_r(\overline{X})]^{\frac{1}{2}}\big)$
  for some $(\overline{R},\overline{Q})\in\mathbb{O}^{n\times r}\times\mathbb{O}^{m\times r}$ such that $\overline{X}=\overline{R}\Sigma_{r}(\overline{X})\overline{Q}^{\top}$.
  \end{description}
 \end{lemma}
 \begin{proof}
 Item (i) is immediate by Corollary \ref{corollary4.1} (ii), so it suffices to prove item (ii). Pick any $W=(\widehat{U},\widehat{V},\overline{U},\overline{V},\widehat{X},\overline{X})\in\mathcal{W}^*$. There exists an index set $\mathcal{K}\subset\mathbb{N}$ such that $\lim_{\mathcal{K}\ni k\to\infty}W^k=W$. From steps 2 and 4 of Algorithm \ref{MPAMSC}, it follows that
 \begin{align}\label{tempU-equa41}
 \!\widehat{U}\!\!=\!\!\lim_{\mathcal{K}\ni k\to\infty} \!\!\widehat{U}^{k+1}\!\!=\!\!\lim_{\mathcal{K}\ni k\to\infty}\!\!\widehat{P}^{k+1}\widehat{D}^{k+1}\ \ {\rm and}\ \ \widehat{V}\!\!=\!\!\lim_{\mathcal{K}\ni k\to\infty}\widehat{V}^{k+1}\!\!=\!\!\lim_{\mathcal{K}\ni k\to\infty}\!\!\widehat{R}^{k+1}\widehat{D}^{k+1},\\
 \label{tempV-equa41}
  \overline{U}\!\!=\!\!\lim_{\mathcal{K}\ni k\to\infty}\!\!\overline{U}^{k+1}\!\!=\!\!\lim_{\mathcal{K}\ni k\to\infty}\!\!\overline{R}^{k+1}\overline{D}^{k+1}\ \ {\rm and}\ \ \overline{V}\!\!=\!\!\lim_{\mathcal{K}\ni k\to\infty}\!\overline{V}^{k+1}\!\!=\!\!\lim_{\mathcal{K}\ni k\to\infty}\!\overline{P}^{k+1}\overline{D}^{k+1}
 \end{align}
 with $\widehat{R}^{k+1}=\overline{P}^{k}\widehat{Q}^{k+1}$ and $\overline{R}^{k+1}=\widehat{P}^{k+1}\overline{Q}^{k+1}$ for each $k\in\mathbb{N}$. Note that $\{\widehat{R}^{k+1}\}_{k\in\mathbb{N}}\subset\mathbb{O}^{m\times r}$ and $\{\widehat{P}^{k+1}\}_{k\in \mathbb{N}}\subset\mathbb{O}^{n\times r}$. By the compactness of $\mathbb{O}^{l\times r}$, there exists an index set $\mathcal{K}_1\subset\mathcal{K}$ such that $\{\widehat{R}^{k+1}\}_{k\in\mathcal{K}_1}$ and $\{\widehat{P}^{k+1}\}_{k\in\mathcal{K}_1}$ are convergent, i.e., there are $\widehat{R}\in \mathbb{O}^{m\times r}$ and $\widehat{P}\!\in\!\mathbb{O}^{n\times r}$ such that
$\lim_{\mathcal{K}_1\ni k\to\infty} \widehat{R}^{k+1}=\widehat{R}$ and $\lim_{\mathcal{K}_1\ni k\to\infty}\widehat{P}^{k+1}=\widehat{P}$. Together with  $\lim_{\mathcal{K}\ni k\to\infty}\widehat{D}^{k+1}=[\Sigma_r(\widehat{X})]^{{1}/{2}}$ and the above \eqref{tempU-equa41}, we have
\begin{equation}\label{equa-widehatUV}
 \widehat{U}\!=\!\!\lim_{\mathcal{K}_1\ni
 k\to\infty}\!\widehat{P}^{k+1}\widehat{D}^{k+1}\!=\!\widehat{P}[\Sigma_{r}(\widehat{X})]^{\frac{1}{2}}\ \ {\rm and}\ \ \widehat{V}\!=\!\!\lim_{\mathcal{K}_1\ni k\to\infty}\!\widehat{R}^{k+1}\widehat{D}^{k+1}\!=\!\widehat{R}[\Sigma_r(\widehat{X})]^{\frac{1}{2}}.
 \end{equation}
 Combining with $\widehat{X}=\lim_{\mathcal{K}\ni k\to\infty}\widehat{X}^k=\lim_{\mathcal{K}\ni k\to\infty}\widehat{U}^k(\widehat{V}^k)^{\top}=\widehat{U}\widehat{V}^{\top}$ results in $\widehat{X}=\widehat{P}\Sigma_{r}(\widehat{X})\widehat{R}^{\top}$.
 By using the above \eqref{tempV-equa41} and the similar arguments, there are $\overline{R}\!\in\!\mathbb{O}^{n\times r}$ and $\overline{Q}\!\in\!\mathbb{O}^{m\times r}$ such that $\overline{U}\!=\!\overline{R}[\Sigma_{r}(\overline{X})]^{{1}/{2}}$ and  $\overline{V}\!=\!\overline{Q}[\Sigma_{r}(\overline{X})]^{{1}/{2}}$. Along with $\overline{X}=\lim_{\mathcal{K}\ni k\to\infty}\overline{X}^k=\lim_{\mathcal{K}\ni k\to\infty}\overline{U}^k(\overline{V}^k)^{\top}=\overline{U}\overline{V}^{\top}$, we obtain $\overline{X}=\overline{R}\Sigma_{r}(\overline{X})\overline{Q}^{\top}$. Using Corollary \ref{corollary4.1} (iv) leads to $0=\lim_{\mathcal{K}\ni k\to\infty}\|\widehat{X}^k-\overline{X}^k\|_F=\|\widehat{X}-\overline{X}\|_F$.
 \end{proof}
 \begin{theorem}\label{Sub-convergence}
  Under Assumption \ref{ass1}, for each $W=(\widehat{U},\widehat{V},\overline{U},\overline{V},\widehat{X},\overline{X})\in\mathcal{W}^*$, $(\overline{U},\overline{V})$ and     $(\widehat{U},\widehat{V})$ are the stationary points of \eqref{prob} and $\Phi_{\lambda,\mu}(\overline{U},\overline{V})=\Phi_{\lambda,\mu}(\widehat{U},\widehat{V})=\varpi^*$.
 \end{theorem}
 \begin{proof}
  Pick any $W=(\widehat{U},\widehat{V},\overline{U},\overline{V},\widehat{X},\overline{X})\in\mathcal{W}^*$. There exists an index set $\mathcal{K}\subset\mathbb{N}$ such that $\lim_{\mathcal{K}\ni k\to\infty}W^k=W$. We first claim that the following two limits hold:
  \begin{equation}\label{vartheta-limit}
   \lim_{\mathcal{K}\ni k\to\infty}\vartheta(U^{k+1})=\vartheta(\overline{U})
   \ \ {\rm and}\ \ \lim_{\mathcal{K}\ni k\to\infty}\vartheta(V^{k+1})=\vartheta(\widehat{V}).
  \end{equation}
  Indeed, for each $k\in\mathbb{N}$, by the definition of $U^{k+1}$ in step 1 and the expression of $\widehat{\Phi}_{\lambda,\mu}$,
 \begin{align*}
  &\widehat{F}(U^{k+1},\overline{V}^k,\overline{U}^k,\overline{V}^k)+
  \frac{\mu}{2}\|{U}^{k+1}\|_F^2+\lambda \vartheta({U}^{k+1})+\frac{\gamma_{1,k}}{2}\|U^{k+1}-\overline{U}^k\|_F^2\\
  &\leq \widehat{F}(\overline{U},\overline{V}^k,\overline{U}^k,\overline{V}^k)
   +\frac{\mu}{2}\|\overline{U}\|_F^2+\lambda \vartheta(\overline{U})
   +\frac{\gamma_{1,k}}{2}\|\overline{U}-\overline{U}^k\|_F^2.
 \end{align*}
 From Corollary \ref{corollary4.1} (i) and $\lim_{\mathcal{K}\ni k\to\infty}\overline{U}^k=\overline{U}$, we have $\lim_{\mathcal{K}\ni k\to\infty}U^{k+1}=\overline{U}$. Now passing the limit $\mathcal{K}\ni k\to\infty$ to the above inequality and using $\lim_{\mathcal{K}\ni k\to\infty}W^k=W$, Corollary \ref{corollary4.1} (i) and the continuity of $\widehat{F}$ results in $\limsup_{\mathcal{K}\ni k\to\infty} \vartheta({U}^{k+1})\le\vartheta(\overline{U})$.
 Along with the lower semicontinuity of $\vartheta$, we obtain $\lim_{ \mathcal{K}\ni k\to\infty}\vartheta({U}^{k+1})=\vartheta(\overline{U})$. Similarly, for each $k\in\mathbb{N}$, by the definition $V^{k+1}$ in step 3,
 \begin{align*}
  &\widehat{F}(\widehat{U}^{k+1},{V}^{k+1},\widehat{U}^{k+1},\widehat{V}^{k+1})+
  \frac{\mu}{2}\|V^{k+1}\|_F^2+\lambda \vartheta(V^{k+1})+\frac{\gamma_{1,k}}{2}\|V^{k+1}-\widehat{V}^{k+1}\|_F^2\\
  &\leq \widehat{F}(\widehat{U}^{k+1},\widehat{V},\widehat{U}^{k+1},\widehat{V}^{k+1}) +\frac{\mu}{2}\|\widehat{V}\|_F^2+\lambda \vartheta(\widehat{V})
   +\frac{\gamma_{1,k}}{2}\|{\widehat{V}^{k+1}}-\widehat{V}\|_F^2.
 \end{align*}
 Following the same arguments as above leads to the second limit in \eqref{vartheta-limit}. 
 
 Let $J=[\overline{r}],\,\overline{J}=[r]\setminus J$ and $\Lambda\!:={\rm Diag}(\theta'(\|\overline{U}_1\|),\ldots,\theta'(\|\overline{U}_{\overline{r}}\|))$. We next prove that
 \begin{subnumcases}{}\label{EgradF-Uk2}	
  0=\nabla f(\overline{X})
  \overline{V}_{\!J}+\mu \overline{U}_{\!J}+\lambda\overline{U}_{\!J}\Sigma_1^{-1/2}\Lambda,\\
 \label{EgradF-Vk2} 	
 0=\nabla f(\widehat{X})^{\top}\widehat{U}_{\!J}+\mu \widehat{V}_J  +\lambda\widehat{V}_{\!J}\Sigma_1^{-1/2}\Lambda
 \end{subnumcases}
 with $\Sigma_1={\rm Diag}(\sigma_1(\overline{X}),\ldots,\sigma_{\overline{r}}(\overline{X}))$.
 Passing the limit $\mathcal{K}\ni k\to\infty$ to the inclusions in Remark \ref{remark-MPAMSC} (b), and using \eqref{vartheta-limit} and Corollary \ref{corollary4.1} (i) and (iv) results in
 \[
  0\in\nabla f(\overline{X})\overline{V}+\mu \overline{U}+\lambda\partial\vartheta(\overline{U})\ \ {\rm and}\ \ 
 0\in\nabla\! f(\widehat{X})^{\top}\widehat{U}+\mu\widehat{V}
  +\lambda\partial\vartheta(\widehat{V}). 
 \]
 By Lemma \ref{Lemma4.1} (ii),  $\overline{U}_{\overline{J}}=\widehat{U}_{\overline{J}}=0$ and $\overline{V}_{\overline{J}}=\widehat{V}_{\overline{J}}=0$. Along with the above two inclusions, using Lemma \ref{subdiff-vtheta} leads to $0\in\partial\phi(0)$, where $\phi$ is the function defined in Lemma \ref{subdiff-vtheta}, and for each $j\in J$,
 \begin{subnumcases}{}\label{gradF-Uk2}
  0=\nabla f(\overline{X})
  \overline{V}_{\!j}+\mu \overline{U}_{j}+\lambda\theta'(\|\overline{U}_{\!j}\|)\|\overline{U}_j\|^{-1}\overline{U}_{j},\\
 \label{gradF-Vk2} 	
 0=\nabla f(\widehat{X})^{\top}\widehat{U}_{j}+\mu \widehat{V}_j  +\lambda\theta'(\|\widehat{V}_{j}\|)\|\widehat{V}_j\|^{-1}\widehat{V}_{j}.
 \end{subnumcases}
 By Lemma \ref{Lemma4.1} (ii), $\|\overline{U}_{\!j}\|=\|\widehat{V}_j\|=\sigma_j(\widehat{X})^{1/2}=\sigma_j(\overline{X})^{1/2}$ for each $j\in J$. Consequently, 
 \begin{equation*}
  {\rm Diag}\big(\|\overline{U}_1\|,\ldots,\|\overline{U}_{\overline{r}}\|\big)=\Sigma_1^{1/2}\ \ {\rm and}\ \
  {\rm Diag}\big(\|\widehat{V}_1\|,\ldots,\|\widehat{V}_{\overline{r}}\|\big)=\Sigma_1^{1/2}.
 \end{equation*}
 Then, for each $j\in J$, the above \eqref{gradF-Uk2}-\eqref{gradF-Vk2} can be compactly written as \eqref{EgradF-Uk2}-\eqref{EgradF-Vk2}. 	
 
 By comparing \eqref{EgradF-Uk2}-\eqref{EgradF-Vk2} with Definition \ref{def-spoint}, and using Lemma \ref{subdiff-vtheta}, $0\in\partial\phi(0)$, and $\overline{U}_{\overline{J}}=\widehat{U}_{\overline{J}}=0$ and $\overline{V}_{\overline{J}}=\widehat{V}_{\overline{J}}=0$, the rest only needs to prove that 
 \begin{subnumcases}{}\label{aim-EgradF-Uk2}	
  0=\nabla f(\widehat{X})
  \widehat{V}_{\!J}+\mu \widehat{U}_{\!J}+\lambda\widehat{U}_{\!J}\Sigma_1^{-1/2}\Lambda,\\
 \label{aim-EgradF-Vk2} 	
 0=\nabla f(\overline{X})^{\top}\overline{U}_{\!J}+\mu \overline{V}_{\!J} +\lambda\overline{V}_{\!J}\Sigma_1^{-1/2}\Lambda,
 \end{subnumcases}
 since \eqref{aim-EgradF-Uk2} and \eqref{EgradF-Vk2} imply that $(\widehat{U},\widehat{V})$ is a stationary point of \eqref{prob}, and while \eqref{aim-EgradF-Vk2} and \eqref{EgradF-Uk2} imply that $(\overline{U},\overline{V})$ is a stationary point of \eqref{prob}. By Lemma \ref{Lemma4.1} (ii), 
 \begin{equation}\label{relation-temp}
 \Sigma_r(\widehat{X})=\Sigma_r(\overline{X})\!:=\Sigma_r\ {\rm and}\ \widehat{P}\Sigma_r\widehat{R}^{\top}\!=\widehat{U}\widehat{V}^{\top}\!=\widehat{X}=\overline{X}=\overline{U}\overline{V}^{\top}=\overline{R}\Sigma_r\overline{Q}^{\top}. 
 \end{equation}
 Recall that $\widehat{P},\overline{R}\in\mathbb{O}^{n\times r}$ and $\widehat{R},\overline{Q}\in\mathbb{O}^{m\times r}$. There exist $\widehat{P}^{\perp},\overline{R}^{\perp}\in\mathbb{O}^{n\times(n-r)}$ and $\widehat{R}^{\perp},\overline{Q}^{\perp}\in\mathbb{O}^{m\times(m-r)}$ such that $[\widehat{P}\ \ \widehat{P}^{\perp}],[\overline{R}\ \ \overline{R}^{\perp}]\in\mathbb{O}^n$ and $[\widehat{R}\ \ \widehat{R}^{\perp}],[\overline{Q}\ \ \overline{Q}^{\perp}]\in\mathbb{O}^m$. Moreover, in view of the second group of equalities in \eqref{relation-temp}, it holds
  \[
   \big[\widehat{P}\ \ \widehat{P}^{\perp}\big]
  \begin{pmatrix}
  \Sigma_r & 0\\ 0 & 0
  \end{pmatrix}\big[\widehat{R}\ \ \widehat{R}^{\perp}\big]^{\top}
  =\big[\overline{R}\ \ \overline{R}^{\perp}\big]
  \begin{pmatrix}
   \Sigma_r & 0\\ 0 & 0
  \end{pmatrix}\big[\overline{Q}\ \ \overline{Q}^{\perp}\big]^{\top}.
  \]
  Let $\overline{\mu}_1>\cdots>\overline{\mu}_{\kappa}$ be the distinct singular values of $\overline{X}=\widehat{X}$. For each $l\in[\kappa]$, write $\alpha_l\!:=\!\{j\in[n]\ |\ \sigma_j(\overline{X})\!=\!\overline{\mu}_l\}$. From the above equality and \cite[Proposition 5]{Ding14}, there is a block diagonal orthogonal $\widetilde{Q}={\rm BlkDiag}(\widetilde{Q}_1,\ldots,\widetilde{Q}_{\kappa})$ with $\widetilde{Q}_l\in\mathbb{O}^{|\alpha_l|}$ such that
  \begin{equation}\label{relation-temp2}
  \big[\widehat{P}\ \ \widehat{P}^{\perp}\big]=\big[\overline{R}\ \ \overline{R}^{\perp}\big]\widetilde{Q},\ \big[\widehat{R}\ \ \widehat{R}^{\perp}\big]=\big[\overline{Q}\ \ \overline{Q}^{\perp}\big]\widetilde{Q}\ \ {\rm and}\ \
  \widetilde{Q}\begin{pmatrix}
  \Sigma_r & 0\\ 0 & 0
  \end{pmatrix}=
  \begin{pmatrix}
  \Sigma_r & 0\\ 0 & 0
  \end{pmatrix}\widetilde{Q}.
  \end{equation}
 Let $\widetilde{Q}^{1}={\rm BlkDiag}(\widetilde{Q}_1,\ldots,\widetilde{Q}_{\kappa-1})\in\mathbb{O}^{\overline{r}}$, and let $\widehat{P}^1$ and $\widehat{R}^1$ be the matrices consisting of the first $\overline{r}$ columns of $\widehat{P}$ and $\widehat{R}$. Then $\widehat{P}^1=\overline{R}\widetilde{Q}^{1}$ and $\widehat{R}^1=\overline{Q}\widetilde{Q}^{1}$. Also, from the third equality in \eqref{relation-temp2} and $\Sigma_1={\rm Diag}(\sigma_1(\overline{X}),\ldots,\sigma_{\overline{r}}(\overline{X}))$, we have $\widetilde{Q}^{1}\Sigma_{1}=\Sigma_{1}\widetilde{Q}^{1}$. Along with $\widehat{U}=\widehat{P}\Sigma_{r}^{1/2},\overline{U}=\overline{R}\Sigma_{r}^{1/2}$ and $\widehat{V}=\widehat{R}\Sigma_{r}^{1/2},\overline{V}=\overline{Q}\Sigma_r^{1/2}$ by Lemma \ref{Lemma4.1} (ii), it holds $\widehat{U}_{J}=\widehat{P}^1\Sigma_{1}^{1/2}=\overline{R}\widetilde{Q}^{1}\Sigma_{1}^{1/2}=\overline{R}\Sigma_{1}^{1/2}\widetilde{Q}^{1}=\overline{U}_{\!J}\widetilde{Q}^1$ and $\widehat{V}_{\!J}=\widehat{R}^1\Sigma_{1}^{1/2}=\overline{Q}\widetilde{Q}^{1}\Sigma_{1}^{1/2}=\overline{Q}\Sigma_{1}^{1/2}\widetilde{Q}^{1}=\overline{V}_{\!J}\widetilde{Q}^1$. Substituting $\widehat{U}_{J}=\overline{U}_{\!J}\widetilde{Q}^1, \widehat{V}_{\!J}=\overline{V}_{\!J}\widetilde{Q}^1$ and $\overline{X}=\widehat{X}$ into \eqref{EgradF-Uk2}-\eqref{EgradF-Vk2} yields 
 \begin{subnumcases}{}\label{EgradF-Uk3}	
  \nabla f(\widehat{X})\widehat{V}_{\!J}+\mu \widehat{U}_{\!J} + \lambda\overline{U}_{\!J}\Sigma_1^{-\frac{1}{2}}\Lambda\widetilde{Q}^1=0,\\
  \label{EgradF-Vk3}
  \nabla f(\overline{X})^{\top}\overline{U}_{\!J}+\mu \overline{V}_{\!J} + \lambda\widehat{V}_J\Sigma_1^{-\frac{1}{2}}\Lambda(\widetilde{Q}^1)^{\top}=0.
  \end{subnumcases}
 Recall that $\|\overline{U}_{\!j}\|=\|\widehat{V}_j\|=\sigma_j(\widehat{X})^{1/2}=\sigma_j(\overline{X})^{1/2}$ for each $j\in J$. By the expressions of $\widetilde{Q}^1$ and $\Lambda$, we have $\Lambda \widetilde{Q}^1=\widetilde{Q}^1\Lambda$ and $\Lambda (\widetilde{Q}^1)^{\top}=(\widetilde{Q}^1)^{\top}\Lambda$.
Then, it holds that 
\[
 \Sigma_1^{-\frac{1}{2}}\Lambda\widetilde{Q}^1=\!\Sigma_1^{-\frac{1}{2}}\widetilde{Q}^1\Lambda=\!\widetilde{Q}^1\Sigma_1^{-\frac{1}{2}}\Lambda\ \ {\rm and}\ \ \Sigma_1^{-\frac{1}{2}}\Lambda(\widetilde{Q}^1)^{\top}\!=\Sigma_1^{-\frac{1}{2}}(\widetilde{Q}^1)^{\top}\Lambda\!=(\widetilde{Q}^1)^{\top}\Sigma_1^{-\frac{1}{2}}\Lambda. 
 \]
 Together with equations \eqref{EgradF-Uk3}-\eqref{EgradF-Vk3}, $\widehat{U}_{J}=\overline{U}_{\!J}\widetilde{Q}^1$ and $\widehat{V}_J=\overline{V}_{\!J} \widetilde{Q}^1$, we obtain the desired \eqref{aim-EgradF-Uk2}-\eqref{aim-EgradF-Vk2}, and the first part of the conclusions follows. 
 
 For the second part, using Lemma \ref{Lemma4.1} (ii) and the expression of $\Phi_{\lambda,\mu}$ leads to $\Phi_{\lambda,\mu}(\overline{U},\overline{V})=\Phi_{\lambda,\mu}(\widehat{U},\widehat{V})$. By the continuity of $\widehat{F}$ and \eqref{vartheta-limit}, using Theorem \ref{converge-obj} yields $\Phi_{\lambda,\mu}(\overline{U},\overline{V})=\varpi^*=\Phi_{\lambda,\mu}(\widehat{U},\widehat{V})$. The result holds by the arbitrariness of $W\in \mathcal{W}^*$.
\end{proof}
\begin{remark}\label{remark1-converge}
 Theorems \ref{converge-obj} and \ref{Sub-convergence} respectively provide the convergence certificate of the objective value sequences and the subsequence convergence guarantee for Algorithm \ref{MPAMSC} to solve \eqref{prob} with $\theta$ from one of $\theta_1$-$\theta_4$ in Table \ref{tab1}. When $\theta=\theta_2$, these results provide the corresponding theoretical certificates for softImpute-ALS \cite[Algorithm 3.1]{Hastie15}. 
\end{remark}
\subsection{Full convergence of iterate sequence}\label{sec4.3}
 
 Next we focus on the full convergence of the iterate sequence under additional Assumption \ref{ass2}, which imposes a restriction on the proximal mapping of $\theta$ associated with $\gamma\in[\lambda\widehat{\beta}^{-1},\lambda(\mu\!+\!\underline{\gamma})^{-1}]$. It is easy to check that Assumption \ref{ass2} holds for $\theta=\theta_1$ and $\theta_4$. 
 \begin{assumption}\label{ass2}
  Fix any $\gamma\in[\lambda\widehat{\beta}^{-1},\lambda(\mu\!+\!\underline{\gamma})^{-1}]$ where $\widehat{\beta}$ is the same as in Remark \ref{remark41}. There exists $c_p>0$ such that for any $t\ge 0$, either $\mathcal{P}_{\!\gamma}\theta(t)=\{0\}$ or $\min_{s\in\mathcal{P}_{\!\gamma}\theta(t)}s\ge c_p$.
 \end{assumption}

 The following lemma shows that under Assumptions \ref{ass1}-\ref{ass2}, the rank function sequences $\{{\rm rank}(\widehat{X}^k)\}_{k\in\mathbb{N}}$ and $\{{\rm rank}(\overline{X}^k)\}_{k\in\mathbb{N}}$ converge to the same value.
 \begin{lemma}\label{rank-prop}
  Consider any $W=\!(\widehat{U},\widehat{V},\overline{U},\overline{V},\widehat{X},\overline{X})\in\mathcal{W}^*$. Let $\overline{k}$ be the same as in Proposition \ref{prop-nzind}. Then, under Assumptions \ref{ass1}-\ref{ass2}, 
  \begin{description}
  \item[(i)] when $k\ge\overline{k}$, ${\rm rank}(\widehat{X}^{k})={\rm rank}(\overline{X}^{k})=\|\widehat{U}^{k}\|_{2,0}=\|\overline{U}^{k}\|_{2,0}=\|\overline{U}\|_{2,0}$;

  \item[(ii)] when $k\ge\overline{k}$, $\max\{\sigma_{i}(\overline{X}^k),\sigma_{i}(\widehat{X}^k)\}\!=\!0$ for $i=\overline{r}\!+\!1,\ldots,r$;

  \item[(iii)] there exist $\alpha>0$ and $\widetilde{k}\!\ge\overline{k}$ such that $\min\{\sigma_{\overline{r}}(\overline{X}^k),\sigma_{\overline{r}}(\widehat{X}^k)\}\ge\alpha$ for $k\ge\widetilde{k}$.
  \end{description}  
 \end{lemma}
 \begin{proof}
 {\bf(i)-(ii)} By Proposition \ref{prop-nzind}, for all $k\ge\overline{k}$, $J_{\overline{U}^{k+1}}=J_{\overline{U}^{k}}=J_{U^k}\!:=J$, so  $\overline{U}^{k+1}=[\overline{U}_{\!J}^{k+1}\ 0]$. 
 Since $W\in\mathcal{W}^*$, there exists an index set $\mathcal{K}\subset\mathbb{N}$ such that $\lim_{\mathcal{K}\ni k\to\infty}\overline{U}^k=\overline{U}$, which by the first limit in Corollary \ref{corollary4.1} (i) implies that $\lim_{\mathcal{K}\ni k\to\infty}U^{k+1}=\overline{U}$. 
 Combining equation \eqref{Uk-def} in Remark \ref{remark-MPAMSC} (c) and Assumption \ref{ass2} leads to $\min_{i\in J}\big\|[{U}_{\!J}^{k+1}]_i\big\|\ge c_p>0$ for all $k\ge\overline{k}$, which means that $\|\overline{U}\|_{2,0}\ge|J|$. In addition, it is clear that $\|\overline{U}\|_{2,0}\le |J|$. Thus, $J=\|\overline{U}\|_{2,0}$. The conclusions of item (i) and (ii) then follow \eqref{rank-equal}.

 \noindent
 {\bf(iii)} Suppose on the contrary that the conclusion does not hold. There exists an index set $\mathcal{K}\subset\mathbb{N}$ such that $\lim_{\mathcal{K}\ni k\to\infty}\sigma_{\overline{r}}(\overline{X}^k)=0$. By the continuity of $\sigma_{\overline{r}}(\cdot)$, the sequence $\{\overline{X}^k\}_{k\in\mathbb{N}}$ has a cluster point, say $\widetilde{X}$, satisfying ${\rm rank}(\widetilde{X})\le\overline{r}-1$, which is a contradiction to the result of item (i). Thus, we complete the proof. 
 \end{proof}

 Next we apply the $\sin\Theta$ theorem of \cite{Dopico00} to establish a crucial property of $\{W^k\}_{k\in\mathbb{N}}$, which will be used later to control the distance ${\rm dist}(0,\partial\Phi_{\lambda,\mu}(\overline{U}^k,\overline{V}^k))$.
 \begin{lemma}\label{lemma-UVk}
 For each $k$, let $\widehat{D}_1^{k}\!\!=\!{\rm Diag}(\widehat{D}_{11}^{k},\ldots,\!\widehat{D}_{\overline{r}\overline{r}}^{k})$ and $\overline{D}_1^{k}\!\!=\!{\rm Diag}(\overline{D}_{11}^{k},\ldots,\!\overline{D}_{\overline{r}\overline{r}}^{k})$. Then, under Assumptions \ref{ass1}-\ref{ass2}, for each $k\ge\widetilde{k}$, there exist $R_1^k,R_2^k\in\mathbb{O}^{\overline{r}}$ such that with $A^k={\rm BlkDiag}\big((\widehat{D}_1^k)^{-1}R_1^k\widehat{D}_1^k,0\big)\in\mathbb{R}^{r\times r}$ and $B^k={\rm BlkDiag}\big((\overline{D}_1^k)^{-1}R_2^k\overline{D}_1^k,0\big)\in\mathbb{R}^{r\times r}$,
 \begin{align*}
 &\max\big\{\|\overline{U}^{k+1}\!-\widehat{U}^{k+1}A^{k+1}\|_F,\|\overline{V}^{k+1}\!-\widehat{V}^{k+1}A^{k+1}\|_F\big\}\le\frac{\sqrt{\alpha}+4\overline{\beta}}{2\alpha}\big\|\overline{X}^{k+1}\!-\!\widehat{X}^{k+1}\big\|_F,\\
 \label{dist-VV2} 	
 &\max\big\{\|\overline{U}^{k+1}\!-\overline{U}^{k}B^{k}\|_F,\|\overline{V}^{k+1}\!-\overline{V}^{k}B^{k}\|_F\big\}\le\frac{\sqrt{\alpha}+4\overline{\beta}}{2\alpha}\big\|\overline{X}^{k+1}\!-\!\overline{X}^{k}\big\|_F,
 \end{align*}
 where $\alpha$ and $\widetilde{k}$ are the same as in Lemma \ref{rank-prop}  and $\overline{\beta}$ is the same as in Corollary \ref{corollary4.1} (iii).
 \end{lemma}
 \begin{proof}
 Fix any $k\ge\widetilde{k}$. By steps 2 and 4 of Algorithm \ref{MPAMSC},
 $\widehat{X}^{k+1}=\widehat{P}^{k+1}(\widehat{D}^{k+1})^2(\widetilde{Q}^{k+1})^{\top}$ with $\widetilde{Q}^{k+1}=\overline{P}^{k}\widehat{Q}^{k+1}$ and $\overline{X}^{k+1}=\widetilde{R}^{k+1}(\overline{D}^{k+1})^2(\overline{P}^{k+1})^{\top}$ with $\widetilde{R}^{k+1}=\widehat{P}^{k+1}\overline{Q}^{k+1}$. 
 Let $J=[\overline{r}]$ and $\overline{J}=[r]\backslash J$. By Lemma \ref{rank-prop} (ii), $\widehat{D}_{ii}^{k+1}=\overline{D}_{ii}^{k+1}=0$ with $i\in\overline{J}$, so that
 \[
  \widehat{X}^{k+1}=\widehat{P}_{\!J}^{k+1}(\widehat{D}^{k+1}_1)^2(\widetilde{Q}_{\!J}^{k+1})^{\top}\ \
  {\rm and}\ \ \overline{X}^{k+1}=\widetilde{R}_{\!J}^{k+1}(\overline{D}_1^{k+1})^2(\overline{P}_{\!J}^{k+1})^{\top}.
 \]
Invoking \cite[Theorem 2.1]{Dopico00} with $(A,\widetilde{A})=((\widehat{X}^{k+1})^{\top},(\overline{X}^{k+1})^{\top})$ and $((\overline{X}^{k})^{\top}, (\overline{X}^{k+1})^{\top})$ respectively, there exist $R_1^{k+1}\in \mathbb{O}^{\overline{r}}$ and $R_2^{k}\in\mathbb{O}^{\overline{r}}$ such that
 \begin{align}\label{dist-X1}	
\!\!\sqrt{\|\widehat{P}_{\!J}^{k+1}R_1^{k+1}\!-\!\widetilde{R}_{\!J}^{k+1}\|_F^2\!+\!
  \|\widetilde{Q}_{{J}}^{k+1}R_1^{k+1}\!-\!\overline{P}_{\!J}^{k+1}\|_F^2}\le\frac{2}{\sigma_{\overline{r}}(\overline{X}^{k+1})}\big\|\overline{X}^{k+1}\!-\!\widehat{X}^{k+1}\big\|_F,\\
 \label{dist-X2} 	
 \sqrt{\|\widetilde{R}_{\!J}^{k}R_2^{k}-\widetilde{R}_{\!J}^{k+1}\|_F^2+
 \|\overline{P}_{\!J}^{k}R_2^{k}-\overline{P}_{\!J}^{k+1}\|_F^2}\le\frac{2}{\sigma_{\overline{r}}(\overline{X}^{k\!+\!1})}\big\|\overline{X}^{k}\!-\overline{X}^{k\!+\!1}\big\|_F.\qquad
 \end{align}
 From the expressions of $\overline{U}^{k+1}$ and $\widehat{U}^{k+1}$ (see steps 2 and 4 of Algorithm \ref{MPAMSC}) and that of $A^{k+1}$, we have
 \begin{align}\label{boundUhat}
 \big\|\overline{U}^{k+1}\!-\!\widehat{U}^{k+1}A^{k+1}\big\|_F
 &=\big\|\widetilde{R}_{\!J}^{k+1}\overline{D}_1^{k+1}\!-\!\widehat{P}_{\!J}^{k+1}R_1^{k+1}\widehat{D}_1^{k+1}\big\|_F\nonumber\\
 &\le\|\overline{D}_1^{k+1}-\widehat{D}_1^{k+1}\|_F+\|\widetilde{R}_{J}^{k+1}-\widehat{P}^{k+1}_{{J}}R_1^{k+1}\|_F\|\widehat{D}_1^{k+1}\|\nonumber\\
 &\stackrel{\eqref{dist-X1}}{\le}\sqrt{\textstyle{\sum_{i=1}^{\overline{r}}}\big[\sigma_i(\overline{X}^{k+1})^{1/2}-\sigma_i(\widehat{X}^{k+1})^{1/2}\big]^2}\nonumber\\
 &\quad\ +\frac{2\overline{\beta}}{\sigma_{\overline{r}}(\overline{X}^{k\!+\!1})}\|\overline{X}^{k+1}-\widehat{X}^{k+1}\|_F\nonumber\\
 &\le \frac{1}{2\sqrt{\alpha}}\|\overline{X}^{k+1}\!-\widehat{X}^{k+1}\|_F+\frac{2\overline{\beta}}{\alpha}\|\overline{X}^{k+1}-\widehat{X}^{k+1}\|_F\nonumber\\
 &=\frac{\sqrt{\alpha}+4\overline{\beta}}{2\alpha}\|\overline{X}^{k+1}-\widehat{X}^{k+1}\|_F.
\end{align}
Similarly, by the expressions of $\overline{V}^{k+1}$ and $\widehat{V}^{k+1}$ (see steps 2 and 4 of Algorithm \ref{MPAMSC}),
\begin{align}\label{boundVhat}
 \|\overline{V}^{k+1}\!-\!\widehat{V}^{k+1}A^{k+1}\|_F
 &=\big\|\overline{P}_{\!J}^{k+1}\overline{D}_{\!J}^{k+1}-\widetilde{Q}_{\!J}^{k+1}R_1^{k+1}\widehat{D}_1^{k+1}\big\|_F\nonumber\\
 &\le\|\overline{D}_1^{k+1}-\widehat{D}_1^{k+1}\|_F+\|\widehat{D}_1^{k+1}\|\|\overline{P}_{\!J}^{k+1}-\widetilde{Q}_{\!J}^{k+1}R_1^{k+1}\|_F\nonumber\\
 &\stackrel{\eqref{dist-X1}}{\le}\frac{\sqrt{\alpha}+4\overline{\beta}}{2\alpha}\|\overline{X}^{k+1}-\widehat{X}^{k+1}\|_F.
\end{align}
Inequalities \eqref{boundUhat}-\eqref{boundVhat} imply that the first inequality holds.
 Using inequality \eqref{dist-X2} and following the same arguments as those for \eqref{boundUhat}-\eqref{boundVhat} leads to the second one.
 \end{proof}
\begin{proposition}\label{subdiff-gap}
 Let $J=[\overline{r}]$ and $\overline{J}=[r]\backslash J$. Let $\vartheta_{\!J}$ and $\vartheta_{\!\overline{J}}$ be defined by \eqref{def-vthetaJ} with such $J$ and $\overline{J}$. Suppose that Assumptions \ref{ass1}-\ref{ass2} hold, and that there exists a constant $c>0$ such that for all $k\ge\widetilde{k}$, with $A_{1}^k\!=(\widehat{D}_1^k)^{-1}R_1^k\widehat{D}_1^k$ and $B_{1}^k\!=(\overline{D}_1^k)^{-1}R_2^k\overline{D}_1^k$ it holds 
 \begin{subnumcases}{}\label{ass3-Uk}
 \|\nabla\vartheta_{\!J}(\overline{U}_{\!J}^{k+1})-\nabla\vartheta_{\!J}(U_{\!J}^{k+1})B_1^k\|_F
  \le c\|\overline{U}_{\!J}^{k+1}-{U}_{\!J}^{k+1}B_1^k\|_F,\\
  \label{ass3-Vk}
 \|\nabla\vartheta_{\!J}(\overline{V}_{\!J}^{k+1})-\nabla\vartheta_{\!J}(V_{\!J}^{k+1})A_1^{k+1}\|_F
  \le c\|\overline{V}_{\!J}^{k+1}-{V}_{\!J}^{k+1}A_1^{k+1}\|_F
 \end{subnumcases}
 where $R_1^k$ and $R_2^k$ are the same as in Lemma \ref{lemma-UVk} and  $\widetilde{k}$ is the same as in Lemma \ref{rank-prop} (iii). Then, there exists $c_s>0$ such that for all $k\ge\widetilde{k}$, 
 \begin{align*}
 {\rm dist}\big(0,\partial \Phi_{\lambda,\mu}(\overline{U}^{k+1},\overline{V}^{k+1})\big)
 &\le c_{s}(\|\overline{X}^{k+1}\!-\!\overline{X}^{k}\|_F+\|\widehat{X}^{k+1}\!-\!\overline{X}^{k+1}\|_F)\nonumber\\
 &\ +c_{s}(\|U^{k+1}\!-\!\overline{U}^k\|_F+\|V^{k+1}\!-\!\widehat{V}^{k+1}\|_F).
\end{align*}
\end{proposition}
\begin{proof}
 Fix any $k\ge\widetilde{k}$. From the inclusions in Remark \ref{remark-MPAMSC} (b) and Lemma \ref{subdiff-vtheta}, it holds that
 \begin{subnumcases}{}
 0\in\big[\nabla\!f(\overline{X}^{k})\!+\!L_{\!f}(\widehat{X}^{k+1}\!-\!\overline{X}^{k})\big]
  \overline{V}^{k}+\mu U^{k+1}+\gamma_{1,k}(U^{k+1}\!-\!\overline{U}^k)\nonumber\\
  \quad\ \ +\lambda\big[\{\nabla\vartheta_{J}(U_{\!J}^{k+1})\}\times\partial\vartheta_{\overline{J}}(U_{\overline{J}}^{k+1})\big],\nonumber\\
 0\in\big[\nabla\!f(\widehat{X}^{k+1})+L_{\!f}(\overline{X}^{k+1}\!-\!\widehat{X}^{k+1})\big]^{\top}\widehat{U}^{k+1}+\mu {V}^{k+1}+\gamma_{2,k}({V}^{k+1}\!-\!\widehat{V}^{k+1})\nonumber\\
 \quad\ \ +\lambda\big[\{\nabla\vartheta_{J}(V_{\!J}^{k+1})\}\times\partial\vartheta_{\overline{J}}(V_{\!\overline{J}}^{k+1})\big].\nonumber
 \end{subnumcases}
 By Lemma \ref{rank-prop} (i), $\overline{U}^k_{\overline{J}}={U}^k_{\overline{J}}=0$ and $\overline{V}^k_{\overline{J}}={V}^k_{\overline{J}}=0$.
 Along with the above two inclusions, we have
 $0\in\partial\vartheta_{\overline{J}}(U_{\!\overline{J}}^{k+1}),\, 0\in\partial\vartheta_{\overline{J}}(V_{\!\overline{J}}^{k+1})$ and the following two equations
 \begin{align*}
   0\!=\!\big[\nabla\!f(\overline{X}^{k})\!+\!L_{\!f}(\widehat{X}^{k+1}-\!\overline{X}^{k})\big]
  \overline{V}_{\!J}^{k}+\mu U_{\!J}^{k+1}+\gamma_{1,k}(U_{\!J}^{k+1}\!-\!\overline{U}_{\!J}^k)+\lambda\nabla\vartheta_{J}(U_{\!J}^{k+1}),\qquad\nonumber\\
 0\!=\!\big[\nabla\!f(\widehat{X}^{k+1})\!\!+\!L_{\!f}(\overline{X}^{k+1}\!\!\!-\!\widehat{X}^{k+1})\big]^{\top}\widehat{U}_{\!J}^{k+1}\!+\!\mu {V}_{\!J}^{k+1}\!\!+\!\gamma_{2,k}({V}_{\!J}^{k+1}\!\!-\!\widehat{V}_{\!J}^{k+1})\!+\!\lambda\nabla\vartheta_{J}(V_{\!J}^{k+1}).\nonumber
 \end{align*}
 Multiplying the first equality by $B_1^k$ and the second one by $A_1^{k+1}$ leads to
 \begin{subnumcases}{}\label{temp-subdiffU}
 0=\big[\nabla f(\overline{X}^{k})+L_{\!f}(\widehat{X}^{k+1}\!-\!\overline{X}^{k})\big]
  \overline{V}_{\!J}^{k}B_1^k+\mu U_{\!J}^{k+1}B_1^k\nonumber\\
  \qquad+\gamma_{1,k}(U_{\!J}^{k+1}-\overline{U}_{\!J}^k)B_1^k +\lambda\nabla\vartheta_{\!J}(U_{\!J}^{k+1})B_1^k,\\
  \label{temp-subdiffV}
 0=\big[\nabla f(\widehat{X}^{k+1})\!\!+\!L_{\!f}(\overline{X}^{k+1}\!\!\!-\!\widehat{X}^{k+1})\big]^{\top}\widehat{U}_{\!J}^{k+1}A_1^{k+1}\!\!+\!\mu {V}_{\!J}^{k+1}A_1^{k+1} \nonumber\\
 \qquad+\gamma_{2,k}({V}_{\!J}^{k+1}\!\!-\!\widehat{V}_{\!J}^{k+1})A_1^{k+1}+\lambda\nabla\vartheta_{\!J}(V_{\!J}^{k+1})A_1^{k+1}.
\end{subnumcases}
 Let $S^{k+1}\!:=\lambda\nabla\vartheta_{J}(\overline{U}_{\!J}^{k+1})+\!\mu\overline{U}_{\!J}^{k+1}+\nabla\!f(\overline{X}^{k+1})\overline{V}_{\!J}^{k+1}$ and $T^{k+1}\!:=\lambda\nabla\vartheta_{J}(\overline{V}_{\!J}^{k+1})+\mu\overline{V}_{\!J}^{k+1}+\big[\nabla\! f(\overline{X}^{k+1})\big]^{\top}\overline{U}_{\!J}^{k+1}$. 
 Along with $0\in\partial\vartheta_{\overline{J}}(U_{\!\overline{J}}^{k+1}),0\in\partial\vartheta_{\overline{J}}(V_{\!\overline{J}}^{k+1})$ and Definition \ref{def-spoint}, we have 
 $(S^{k+1},T^{k+1})\in\partial \Phi_{\lambda,\mu}(\overline{U}^{k+1},\overline{V}^{k+1})$, and consequently, 
 \begin{equation}\label{temp-ineq46}
  {\rm dist}\big(0,\partial \Phi_{\lambda,\mu}(\overline{U}^{k+1},\overline{V}^{k+1})\big)\leq\|S^{k+1}\|_F+\|T^{k+1}\|_F.
 \end{equation}
 By the above \eqref{temp-subdiffU}-\eqref{temp-subdiffV}, the matrices $S^{k+1}$ and $T^{k+1}$ are equivalently written as
 \begin{align*}
  S^{k+1}&=\nabla\!f(\overline{X}^{k+1})\overline{V}_{\!J}^{k+1}-\big[\nabla\! f(\overline{X}^{k})+L_{\!f}(\widehat{X}^{k+1}\!-\!\overline{X}^{k})\big] \overline{V}_{\!J}^{k}B_1^k\!+\lambda\nabla\vartheta_{\!J}(\overline{U}_{{J}}^{k+1})\nonumber\\
  &\quad\ -\lambda\nabla\vartheta_{\!J}({U}_{\!J}^{k+1})B_1^k+\mu(\overline{U}_{\!J}^{k+1}- U_{\!J}^{k+1}B_1^k)-\gamma_{1,k}(U_{\!J}^{k+1}-\overline{U}_{\!J}^k)B_1^k,\nonumber\\
 T^{k+1}&=\!\nabla\!f(\overline{X}^{k+1})^{\top}\overline{U}_{\!J}^{k+1}-\big[\nabla\! f(\widehat{X}^{k+1})\!+\!L_f(\overline{X}^{k+1}-\!\widehat{X}^{k+1})\big]^{\top}\widehat{U}_{\!J}^{k+1}A_1^{k+1}\!+\!\lambda\nabla\vartheta_{\!J}(\overline{V}_{\!J}^{k+1})\nonumber\\
 &\quad\ -\lambda\nabla\vartheta_{\!J}({V}_{\!J}^{k+1})A_1^{k+1}+\mu (\overline{V}_{\!J}^{k+1}-V_{\!J}^{k+1}A_1^{k+1})
 -\gamma_{2,k}({V}_{\!J}^{k+1}-\widehat{V}_{\!J}^{k+1})A_1^{k+1}.\nonumber
 \end{align*}
 Note that the matrices $A^k$ and $B^k$ appearing in Lemma \ref{lemma-UVk} satisfy $A^k={\rm BlkDiag}(A_1^k,0)$ and $B^k={\rm BlkDiag}(B_1^k,0)$. Combining the above two equations with $\overline{U}^k_{\overline{J}}={U}^k_{\overline{J}}=0$ and $\overline{V}^k_{\overline{J}}={V}^k_{\overline{J}}=0$ and using $\max(\|A^{k+1}\|,\|B^k\|)\le{\overline{\beta}}/{\sqrt{\alpha}}$, we obtain 
 \begin{align}\label{bound-S}
  \|S^{k+1}\!\|_F
 &\le\|\nabla\!f(\overline{X}^{k+1}\!)\overline{V}^{k+1}\!-\!\nabla\! f(\overline{X}^{k})\overline{V}^{k}B^k\|_F\!+\!L_{\!f}\overline{\beta}\|\widehat{X}^{k+1}\!\!-\!\overline{X}^{k}\!\|_F\nonumber\\
 &\quad+\!\mu\|\overline{U}^{k+1}\!-\!U^{k+1}B^k\|_F +\gamma_{1,k}(\overline{\beta}/{\sqrt{\alpha}})\|U^{k+1}\!-\!\overline{U}^k\|_F\nonumber\\
 &\quad +\lambda\|\nabla\vartheta_{\!J}(U_{\!J}^{k+1})B_1^k-\nabla\vartheta_{\!J}(\overline{U}_{\!J}^{k+1})\|_F,\\
  \label{bound-Gam}
 \|T^{k+1}\|_F
&\le\|\nabla\!f(\overline{X}^{k+1})\overline{U}^{k+1}\!-\!\nabla \!f(\widehat{X}^{k+1})\widehat{V}^{k+1}A^{k+1}\|_F+L_{\!f}\overline{\beta}\|\widehat{X}^{k+1}\!-\!\overline{X}^{k+1}\|_F\nonumber\\
  &\quad\ +\mu\|\overline{V}^{k+1}\!-V^{k+1}A^{k+1}\|_F +\gamma_{2,k}(\overline{\beta}/{\sqrt{\alpha}})\|V^{k+1}\!-\widehat{V}^{k+1}\|_F\nonumber\\
  &\quad\ +\lambda\|\nabla\vartheta_{\!J}(V_{\!J}^{k+1})A_1^{k+1}-\nabla\vartheta_{\!J}(\overline{V}_{\!J}^{k+1})\|_F.
 \end{align}
 Recall that $\nabla\!f$ is Lipschitz continuous with modulus $L_{\!f}$. It is easy to obtain that 
 \begin{align*}
  &\|\nabla\!f(\overline{X}^{k+1})\overline{V}^{k+1}-\nabla\!f(\overline{X}^{k})\overline{V}^{k}B^k\|_F\\
  &\le\|\nabla\!f(\overline{X}^{k+1})-\nabla\! f(\overline{X}^{k})\|_F\|\overline{V}^{k+1}\|+\|\nabla \!f(\overline{X}^{k})\|\|\overline{V}^{k+1}- \overline{V}^{k}B^k\|_F\nonumber\\
  &\le L_f\overline{\beta}\|\overline{X}^{k+1}-\overline{X}^{k}\|_F+\frac{c_f(\sqrt{\alpha}\!+\!4\overline{\beta})}{2\alpha}\|\overline{X}^{k+1}\!-\!\overline{X}^{k}\|_F
 \end{align*}
 with $c_{f}\!:=\sup_{k\in\mathbb{N}}\{\|\nabla\!f(\overline{X}^k)\|,\|\nabla\!f(\widehat{X}^k)\|\}$, where the second inequality is by Lemma \ref{lemma-UVk}. For the term $\|\overline{U}^{k+1}\!-U^{k+1}B^k\|_F$ in inequality \eqref{bound-S}, it holds 
 \begin{align*}
  \|\overline{U}^{k+1}\!-U^{k+1}B^k\|_F
  &\le\|\overline{U}^{k+1}\!-\overline{U}^{k}B^k\|_F+\| \overline{U}^{k}B^k\!-U^{k+1}B^k\|_F\\
  &\le\frac{\sqrt{\alpha}\!+\!4\overline{\beta}}{2\alpha}\|\overline{X}^{k+1}\!-\!\overline{X}^{k}\|_F+\frac{\overline{\beta}}{\sqrt{\alpha}}\|\overline{U}^{k}\!-\!U^{k+1}\|_F.
\end{align*}
Combining the above two inequalities with \eqref{bound-S} and the given \eqref{ass3-Uk} leads to
\begin{align}\label{bound1-S}
 \|S^{k+1}\|_F&\le\big[L_f\overline{\beta}+0.5\alpha^{-1}(c_f+\mu+\!\lambda c)(\sqrt{\alpha}\!+\!4\overline{\beta})\big]\|\overline{X}^{k+1}-\overline{X}^{k}\|_F\nonumber\\
 &\quad+(\mu\!+\!\lambda c+\gamma_{1,k})({\overline{\beta}}/{\sqrt{\alpha}})\|U^{k+1}-\overline{U}^k\|_F+L_{\!f}\overline{\beta}\|\widehat{X}^{k+1}\!\!-\!\overline{X}^{k}\!\|_F.
\end{align}
Using inequality \eqref{bound-Gam} and following the same arguments as those for \eqref{bound1-S} yields
\begin{align*}
 \|T^{k+1}\|_F&\le\big[2L_f\overline{\beta}+0.5\alpha^{-1}({c}_f\!+\!\mu+\lambda c)(\sqrt{\alpha}+4\overline{\beta})\big]\|\overline{X}^{k+1}-\!\widehat{X}^{k+1}\|_F\\
 &\quad+(\mu+\lambda c+\gamma_{1,k})({\overline{\beta}}/{\sqrt{\alpha}})\|V^{k+1}-\widehat{V}^{k+1}\|_F.
\end{align*}
 The desired result follows by combining the above two inequalities with \eqref{temp-ineq46}.
\end{proof}
\begin{remark}
 When $\theta=\theta_1$, since $\nabla\vartheta_{\!J}(\overline{U}_{\!J}^{k+1})=0$
 and $\nabla\vartheta_{\!J}(U_{\!J}^{k+1})=0$, inequalities \eqref{ass3-Uk}-\eqref{ass3-Vk} automatically hold for all $k\ge\widetilde{k}$, so the conclusion of Proposition \ref{subdiff-gap} holds for $\theta=\theta_1$ under Assumptions \ref{ass1}-\ref{ass2}. When $\theta\ne\theta_1$ but $\theta'$ is locally Lipschitz on $\mathbb{R}_{++}$, by Proposition \ref{prop-sdiffgap}, the conclusion of Proposition \ref{subdiff-gap} also holds under Assumptions \ref{ass1}-\ref{ass2} if in addition there exists some $W\in\mathcal{W}^*$ such that the nonzero singular values of $\overline{X}$ are distinct each other and  $\Phi_{\lambda,\mu}$ has the KL property at $(\overline{U},\overline{V})$.
\end{remark}

 Now we are ready to prove the convergence of the iterate sequence $\{(\overline{X}^k,\widehat{X}^k)\}_{k\in\mathbb{N}}$ and column subspace sequences $\{{\rm col}(\widehat{U}^k),{\rm col}(\widehat{V}^k)\}$ and $\{{\rm col}(\overline{U}^k),{\rm col}(\overline{V}^k)\}$.
 \begin{theorem}\label{gconverge}
  Suppose that $\Phi_{\lambda,\mu}$ is a KL function, that Assumptions \ref{ass1}-\ref{ass2} hold, and that Proposition \ref{subdiff-gap} holds. Then, the sequences $\{\overline{X}^k\}_{k\in\mathbb{N}}$ and $\{\widehat{X}^k\}_{k\in\mathbb{N}}$ converge to the same point, say $X^*$, and $\big(P_1^*(\Sigma_{{r}}(X^*))^{\frac{1}{2}},Q_1^*(\Sigma_{{r}}(X^*))^{\frac{1}{2}}\big)$ is a stationary point of problem \eqref{prob}, where $P_1^*$ and $Q_1^*$ are the matrices obtained by deleting the last $n-r$ and $m-r$ columns of $P^*$ and $Q^*$, respectively, with $(P^*,Q^*)\in\mathbb{O}^{n,m}(X^*)$, and furthermore, 
  \begin{subnumcases}{}\label{aim-subspaceU}
  \lim_{ k\to\infty}{\rm col}(U^{k})=\lim_{ k\to\infty}{\rm col}(\overline{U}^{k})=\lim_{ k\to\infty}{\rm col}(\widehat{U}^{k})={\rm col}(X^*);\\
 \lim_{k\to\infty}{\rm col}(V^{k})=\lim_{ k\to\infty}{\rm col}(\overline{V}^{k})=\lim_{ k\to\infty}{\rm col}(\widehat{V}^{k})={\rm row}(X^*).
 \label{aim-subspaceV}
 \end{subnumcases} 
 where the set convergence is in the sense of Painlev$\acute{e}$-Kuratowski convergence.  
 \end{theorem} 
\begin{proof}
Combining Corollary \ref{corollary4.1} with Proposition \ref{subdiff-gap} and Theorem \ref{Sub-convergence} and following the same arguments as those for \cite[Theorem 1]{Bolte14} yields the first two parts. For the last part, by Proposition \ref{prop-nzind},
 we only need to prove $\lim_{ k\to\infty}{\rm col}(\widehat{U}^{k})={\rm col}(X^*)$, which by \cite[Exercise 4.2]{RW98} is equivalent to proving that 
 \begin{equation}\label{aim-inclusion}
  \mathcal{L}:=\big\{u\in\mathbb{R}^{\overline{r}}\ |\ \lim_{k\to\infty}{\rm dist}(u,{\rm col}(\widehat{U}^k))=0\big\}={\rm col}(X^*).
 \end{equation}
 
 Let $\widetilde{P}^*\Sigma_{\overline{r}}(X^*)(\widetilde{Q}^*)^{\top}$ be the skinny SVD of $X^*$ with $\widetilde{P}^*\!\in\mathbb{O}^{n\times \overline{r}}$ and $\widetilde{Q}^*\!\in \mathbb{O}^{m\times \overline{r}}$. We first argue that ${\rm col}(\widehat{P}_{\!J})= {\rm col}(\widetilde{P}^*)$. By Lemma \ref{rank-prop}, for each $k\ge\widetilde{k}$, $\widehat{D}_{jj}^{k}=\overline{D}_{jj}^{k}=0$ for $j=\overline{r}\!+\!1,\ldots,r$. By step 2 of Algorithm \ref{MPAMSC}, for each $k\ge\widetilde{k}$, $\widehat{U}^k=[\widehat{P}_J^{k}\widehat{D}_1^{k}\ \ 0]$ with $J=[\overline{r}]$ and $\widehat{D}_1^{k}={\rm Diag}(\widehat{D}_{11}^k,\ldots,\widehat{D}_{\overline{r}\overline{r}}^k)$. This means that ${\rm col}(\widehat{U}^{k})={\rm col}(\widehat{P}_{\!J}^{k})$ for all $k\ge\widetilde{k}$. 
 Clearly, ${\rm col}(X^*)={\rm col}(\widetilde{P}^*)$.
 For each $k\ge\widetilde{k}$, let $\widehat{P}_2^k\in\mathbb{O}^{n\times(n-\overline{r})}$ be such that $[\widehat{P}_{\!J}^k\ \ \widehat{P}_2^k]\in\mathbb{O}^n$. Let $Z^*=X^*{X^*}^{\top}$ and $\mathbb{O}^n(Z^*)=\big\{P\in\mathbb{O}^n\,|\, Z^*=P^*{\rm Diag}(\lambda(Z^*))P^{\top}\big\}$, where $\lambda(Z^*)$ is the eigenvalue vector of $Z^*$ arranged in a nonincreasing order. By \cite[Lemma 3]{Chen03}, there exists $\eta>0$ such that for all sufficiently large $k$,
 \begin{align*}
  {\rm dist}\big([\widehat{P}_{\!J}^k\ \ \widehat{P}_2^k],\mathbb{O}^{n}(Z^*)\big)&\leq\eta\|\widehat{X}^k(\widehat{X}^k)^{\top}\!-\!X^*(X^*)^{\top}\!\|_F\\
  &\le\eta(\|\widehat{X}^k\|\!+\!\|X^*\|)\|\widehat{X}^k\!-\!X^*\!\|_F,
 \end{align*}
  which by the convergence of $\{\widehat{X}^k\}_{k\in\mathbb{N}}$ implies 
  $\lim\limits_{ k\to\infty}{\rm dist}\big([\widehat{P}_{\!J}^k\ \widehat{P}_2^k],\mathbb{O}^{n}(Z^*)\big)=0.$
  Let $\mathbb{O}^{\overline{r}}(Z^*)\!:=\{P\in\mathbb{O}^{n\times\overline{r}}\,|\, Z^*=P[\Sigma_{\overline{r}}(X^*)]^2P^{\top}\big\}$. For any cluster point $\widehat{P}_{\!J}$ of $\{\widehat{P}_J^k\}_{k\in\mathbb{N}}$, we have
  ${\rm dist}\big(\widehat{P}_{\!J},\mathbb{O}^{\overline{r}}(Z^*)\big)=0$, and hence ${\rm col}(\widehat{P}_{\!J})= {\rm col}(\widetilde{P}^*)$. 
  
  Now pick any $u\in\mathcal{L}$. Then
  \(
   0=\lim_{k\to\infty}{\rm dist}(u,{\rm col}(\widehat{U}^k))=\!\lim_{k\to\infty}\|\widehat{P}_{\!J}^{k}(\widehat{P}_{\!J}^{k})^{\top}u-u\|.
  \)
  From the boundedness of $\{\widehat{P}_{\!J}^{k}\}$, there exists a cluster point $\widehat{P}_{\!J}$ such that
  $u=\widehat{P}_{\!J}\widehat{P}_{\!J}^{\top}u$, i.e., $u\in{\rm col}(\widehat{P}_{\!J})={\rm col}(\widetilde{P}^*)={\rm col}(X^*)$. This shows that $\mathcal{L}\subset{\rm col}(X^*)$. For the converse inclusion, from
  $\lim_{ k\to\infty} {\rm dist}\big([\widehat{P}_{\!J}^k\ \ \widehat{P}_2^k],\mathbb{O}^{n}(Z^*)\big)=0$, it is not hard to deduce that
   \[
 \lim_{k\to\infty}\Big[\min_{R\in\mathbb{O}^{\overline{r}},\widetilde{P}^*R\in\mathbb{O}^{\overline{r}}(Z^*)}\|\widetilde{P}^*-\widehat{P}_J^kR\|_F\Big]=0,
 \]
 which implies $\lim_{k\to\infty}\|\widehat{P}_{\!J}^k(\widehat{P}_{\!J}^k)^{\top}\!-\!\widetilde{P}^*(\widetilde{P}^*)^{\top}\|_F\!=\!0$. Pick any $u\in{\rm col}(X^*)\!=\!{\rm col}(\widetilde{P}^*)$. We have $u=\widetilde{P}^*(\widetilde{P}^*)^{\top}u$. Then,  $\lim_{k\to\infty}{\rm dist}(u,{\rm col}(\widehat{U}^k))=\lim_{k\to\infty}\|\widehat{P}_{\!J}^k(\widehat{P}_{\!J}^k)^{\top}u-u\|=0$. This shows that $u\!\in\!\mathcal{L}$, and the converse inclusion follows.
\end{proof}
\section{Numerical experiments}\label{sec5}

To validate the efficiency of Algorithm \ref{MPAMSC}, we apply it to compute one-bit matrix completions with noise, and compare its performance with that of PALM\_ls (i.e., Algorithm \ref{LSPALM} in Appendix B), PLAM (i.e., Algorithm 2 without line search), and the Hybrid AMM proposed in \cite{TaoQianPan22} but without line search. All tests are performed in MATLAB 2024b on a laptop computer running on 64-bit Windows Operating System with an Intel(R) Core(TM) i9-13905H CPU 2.60GHz and 32 GB RAM.

\subsection{One-bit matrix completions with noise}\label{sec5.1}

We consider one-bit matrix completion under a uniform sampling scheme, in which the unknown true $M^*\in\mathbb{R}^{n\times m}$ is assumed to be low rank. Instead of observing noisy entries of $M=M^*+E$ directly, where $E$ is a noise matrix with i.i.d. entries, we now observe with error the sign of a random subset of the entries of $M^*$. More specifically, assume that a random sample $\Omega=\{(i_1,j_1),(i_2,j_2),\ldots,(i_N,j_N)\}\subset([n]\times[m])^{N}$ of the index set is drawn i.i.d. with replacement according to a uniform sampling distribution $\mathbb{P}\{(i_t,j_t)=(k,l)\}=\frac{1}{nm}$ with $(k,l)\in[n]\times[m]$ for all $t\in[N]$, and the entries $Y_{ij}$ of a sign matrix $Y$ with $(i,j)\in\Omega$ are observed. Let $\phi\!:\mathbb{R}\to[0,1]$ be a cumulative distribution function of $-E_{11}$. Then, the above observation model can be recast as
\begin{equation}\label{1bit-model}
 Y_{ij}=\left\{\!\begin{array}{ll}
        +1 &\ {\rm with\ probability}\ \phi(M^*_{ij}),\\
       -1 &\ {\rm with\ probability}\ 1-\phi(M^*_{ij}),
    \end{array} \right.
\end{equation}
and observe noisy entries $\{Y_{i_t,j_t}\}_{t=1}^{N}$ indexed by $\Omega$. More details can be found in \cite{Cai13}. Two common choices for the function $\phi$ or the distribution of $\{E_{ij}\}$ are as follows:
\begin{description}
 \item [\bf{(I)}] (Logistic regression/noise): The logistic regression model is described by \eqref{1bit-model} with $\phi(x)=\frac{e^{x}}{1+e^x}$ and $E_{ij}$ i.i.d. obeying the standard logistic distribution.

 \item [\bf{(II)}] (Laplacian noise): In this case, $E_{ij}$ i.i.d. obey a Laplacian distribution \textbf{Laplace} $(0,b)$ with the scale parameter $b>0$, and the function $\phi$ has the following form
 \[
   \phi(x)=\left\{\!\begin{array}{cl}
           \frac{1}{2}\exp(x/b)&\ {\rm if}\ x<0,\\
           1-\frac{1}{2}\exp(-x/b)&\ {\rm if}\ x\geq0.
          \end{array} \right.
  \]
 \end{description}

Given a collection of observations $Y_{\Omega}=\{Y_{i_t,j_t}\}_{t=1}^{N}$ from model \eqref{1bit-model}, the negative log-likelihood function can be written as
\begin{equation*}
 f(X)
  =-\sum_{(i,j)\in\Omega}\Big(\mathbb{I}_{[Y_{ij}=1]}\ln \phi(X_{ij})+\mathbb{I}_{[Y_{ij}=-1]}\ln(1- \phi(X_{ij}))\Big).
\end{equation*}
Under case (I),
for each $(i,j)\in[n]\times[m]$, $[\nabla^2\!f(X)]_{ij}=\frac{\exp(X_{ij})}{(1+\exp(X_{ij}))^2}$ for $X\in\mathbb{R}^{n\times m}$, so $\nabla\!f$ is Lipschitz continuous with $L_{\!f}=1$; while under case (II), for any $X\in\mathbb{R}^{n\times m}$ and each $(i,j)\in[n]\times[m]$, $[\nabla^2\!f(X)]_{ij}=\frac{2\exp(-|x|/b)}{b^2(2-\exp(-|x|/b))^2}$ if $X_{ij}Y_{ij}\ge 0$, otherwise $[\nabla^2\!f(X)]_{ij}=0$. Clearly, for case (II), $\nabla\!f$ is Lipschitz continuous with $L_{\!f}={2}/{b^2}$.

For the subsequent tests, we generate randomly the true matrix $M^*\!=M_{L}^*(M_{R}^*)^{\top}$ of rank $r^*$ with the entries of $M_{L}^*\in\mathbb{R}^{n\times r^*}$ and $M_{R}^*\in\mathbb{R}^{m\times r^*}$ drawn independently from a uniform distribution on $[-\frac{1}{2},\frac{1}{2}]$, and then obtain one-bit observations by adding noise and recording the signs of the resulting values. Among others, the noise obeys the standard logistic distribution for case (I) and the Laplacian distribution $\textbf{Laplace}(0,b)$ for case (II) with $b=2$. The noisy observation entries $Y_{i_t,j_t}$ with $(i_t,j_t)\in\Omega$ are obtained by \eqref{1bit-model}, where the index set $\Omega$ is given by uniform sampling. 

\subsection{Implementation details of Algorithm \ref{MPAMSC}}\label{sec5.2}

 First we take a look at the choice of the parameters in Algorithm \ref{MPAMSC}. From equations \eqref{Uk-def}-\eqref{Vk-def} in Remark \ref{remark-MPAMSC} (c), for fixed $\lambda$ and $\mu$, a smaller $\gamma_{1,0}$ (respectively, $\gamma_{2,0}$) will lead to a smooth change of the iterate $U^k$ (respectively, $V^k$), but the associated subproblems will require a little more running time. As a trade-off, we choose $\gamma_{1,0}=\gamma_{2,0}=10^{-2}$ and $\varrho=0.8$ for the subsequent numerical tests. The parameters $\underline{\gamma_{1}}$ and $\underline{\gamma_{2}}$ are set to be $10^{-8}$. The initial $(U^0,V^0)$ is generated in Matlab command $U^0 ={\rm orth}({\rm randn}(n,r)),V^0={\rm orth}({\rm randn}(m,r))$ with $r\in[n]$ specified in the subsequent tests. We terminate Algorithm \ref{MPAMSC} at the $k$th iterate once $k>k_{\rm max}$ or 
 \begin{align*}
  \frac{\|\overline{U}^k(\overline{V}^k)^{\top}\!\!-\!\overline{U}^{k-1}\!(\overline{V}^{k-1})^{\top}\|_F}
  {\max\{1,\|\overline{U}^{k}\!(\overline{V}^{k})^{\top}\|_F\}}\!\le \!\epsilon_1\ {\rm or}\ 
  \frac{\max_{1\le i\le 9}|\Phi_{\mu,\lambda}(\overline{U}^k\!,\!\overline{V}^k)\!-\!\Phi_{\mu,\lambda}(\overline{U}^{k-i}\!,\!\overline{V}^{k-i})|}{\max\{1,\Phi_{\mu,\lambda}(\overline{U}^k\!,\!\overline{V}^k)\}}\!\le\!\epsilon_2.
 \end{align*}
 Our codes can be downloaded from https://github.com/SCUT-OptGroup/PAM\_SC.
 
 For a fair comparison, PALM and PALM\_ls also start from $(U^0,V^0)$, and use the same stop condition as for Algorithm \ref{MPAMSC}, i.e., they are terminated at the $k$th iterate whenever $k>k_{\rm max}$ or either of the following conditions is satisfied:
 \[
 \frac{\|{U}^{k}({V}^{k})^{\top}\!\!-\!\!{U}^{k-1}({V}^{k-1})^{\top}\!\|_F}
  {\max\{1,\|{U}^{k}({V}^{k})^{\top}\|_F\}}\!\le\!\epsilon_1,\,\frac{\max_{1\le i\le 9}|\Phi_{\mu,\lambda}({U}^k\!,\!{V}^k)\!\!-\!\!\Phi_{\mu,\lambda}({U}^{k-i}\!,\!{V}^{k-i})|}{\max\{1,\Phi_{\mu,\lambda}({U}^k,{V}^k)\}}\!\!\le\!\!\epsilon_2.
\]
The parameters of PALM\_ls are set as $\varrho_1\!=\!\varrho_2\!=\!5,\,\underline{\alpha}\!=\!10^{-10}$ and $\overline{\alpha}\!=\!10^{10}$. 

In addition, for the parameters $\mu$ and $\lambda$ involved in model \eqref{prob}, we always choose $\mu=10^{-8}$ and $\lambda=c_{\lambda}\max_{j\in[m]}\|Y_j\|$ with $c_{\lambda}>0$ specified in the experiments.
\subsection{Influence of spectral norm of $M^*$ on Algorithm \ref{MPAMSC}}\label{sec5.3}

 We test how the spectral norm of the true $M^*$ affects the performance of Algorithm \ref{MPAMSC}. To achieve the true matrices $M^*$ with different spectral norms, we generate randomly a true $M_0^*\in\mathbb{R}^{2000\times 2000}$ of rank $r^*=10$ in the same way as in Section \ref{sec5.1}, and set
 \[
   M^* = \frac{c_{M^*}}{\max\{\|{\rm vec}(M_0^*)\|_{\infty},1\}}M_0^*.
 \]
 The noisy observation $Y$ is achieved by \eqref{1bit-model} under Case I with the sample rate ${\rm SR}=0.4$. Figure \ref{fig0} plots the number of iterations and the running time (in seconds) of Algorithm \ref{MPAMSC} and PALM for solving problem \eqref{prob}, generated randomly as above with $c_{M^*}\in\{10, 30, 80, 100,    150,  200, 300,    400, 500, 800, 1000 \}$, with the associated parameter $c_{\lambda}\in\{3.75,    4.03,   3.85,   3.93,    3.74,    3.86,   4.08 ,  4.01,   3.85,   3.85,  3.85 \}$. The numerical results of the two solvers are obtained under the same stop condition described in Section \ref{sec5.2} with $\epsilon_1=5\times 10^{-4}$, $\epsilon_2=10^{-6}$ and $k_{\max}=1000$.
 \begin{figure}[h]
 \centering
\includegraphics[width=\textwidth]{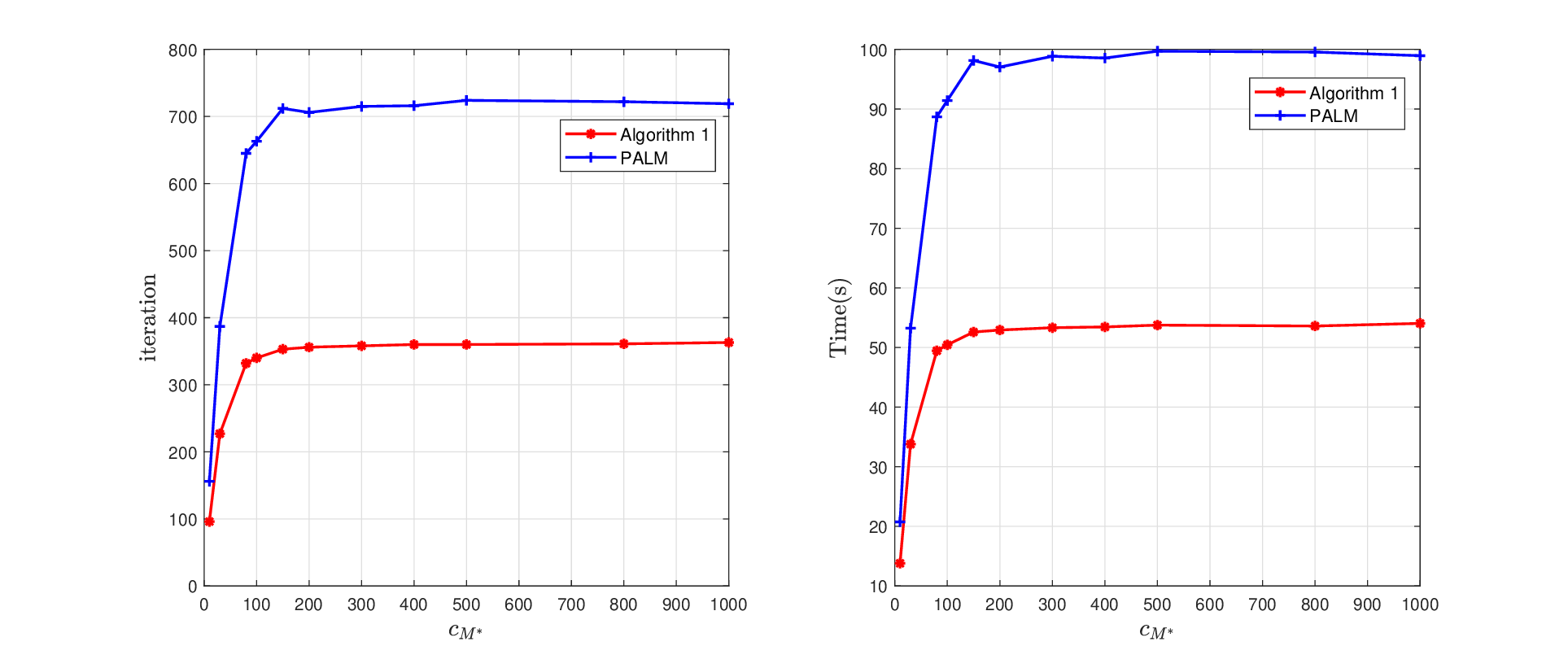}
 \caption{The number of iterations and running time of the two methods with $\theta=\theta_1$}
 \label{fig0}
\end{figure} 

We see that the number of iterations and running time of the two methods increase as the spectral norm of the true $M^*$ becomes larger. However, the number of iterations of PALM has a substantial increase when $c_{M^*}$ increases from $10$ to $150$, and the number of iterations for $c_{M^*}=150$ is about seven times that for $c_{M^*}=10$. This coincides with the remark for PALM in the introduction because, when $M^*$ has a larger spectral norm, it is highly possible for the partial gradients $\nabla_{U}F(\cdot,V^k)$ and $\nabla_{V}F(U^k,\cdot)$ to have larger the Lipschitz modulus. In contrast, the increase of the number of iterations of Algorithm \ref{MPAMSC} is moderate due to the subspace correction, and the number of iterations when $c_{M^*}=150$ is about three times that for $c_{M^*}=10$. This provides a strong support for us to study the majorized PAM method with subspace correction. Note that when $c_{M^*}>150$, the partial gradients $\nabla_{U}F(\cdot,V^k)$ and $\nabla_{V}F(U^k,\cdot)$ almost have the same magnitude as for $c_{M^*}=150$. This interprets why the number of iterations for $c_{M^*}>150$ has a tiny difference from that for $c_{M^*}=150$.  

\subsection{Numerical comparisons with PALMs and Hybrid AMM}\label{sec5.4}

We first test the performance of Algorithm \ref{MPAMSC} in terms of the quality of solutions and the running time, and compare its performance with that of PLAM and PLAM\_ls. The simulated data is generated randomly in the same way as described in Section \ref{sec5.3} with $c_{M^*}=30$, and the noisy observation is obtained by \eqref{1bit-model} under Cases I and II with the sample rate ${\rm SR}=0.4$. Figures \ref{fig01}-\ref{fig02} below plot the objective value and running time curves of the three methods for solving problem \eqref{prob} with $r=10r^*$ and three $\theta$. The numerical results of the three methods are obtained under the same stop condition stated in Section \ref{sec5.2} with  $\epsilon_1=10^{-6},\epsilon_2=10^{-6}$ and $k_{\max}=500$. 
\begin{figure}[h]
 \centering
\includegraphics[width=\textwidth]{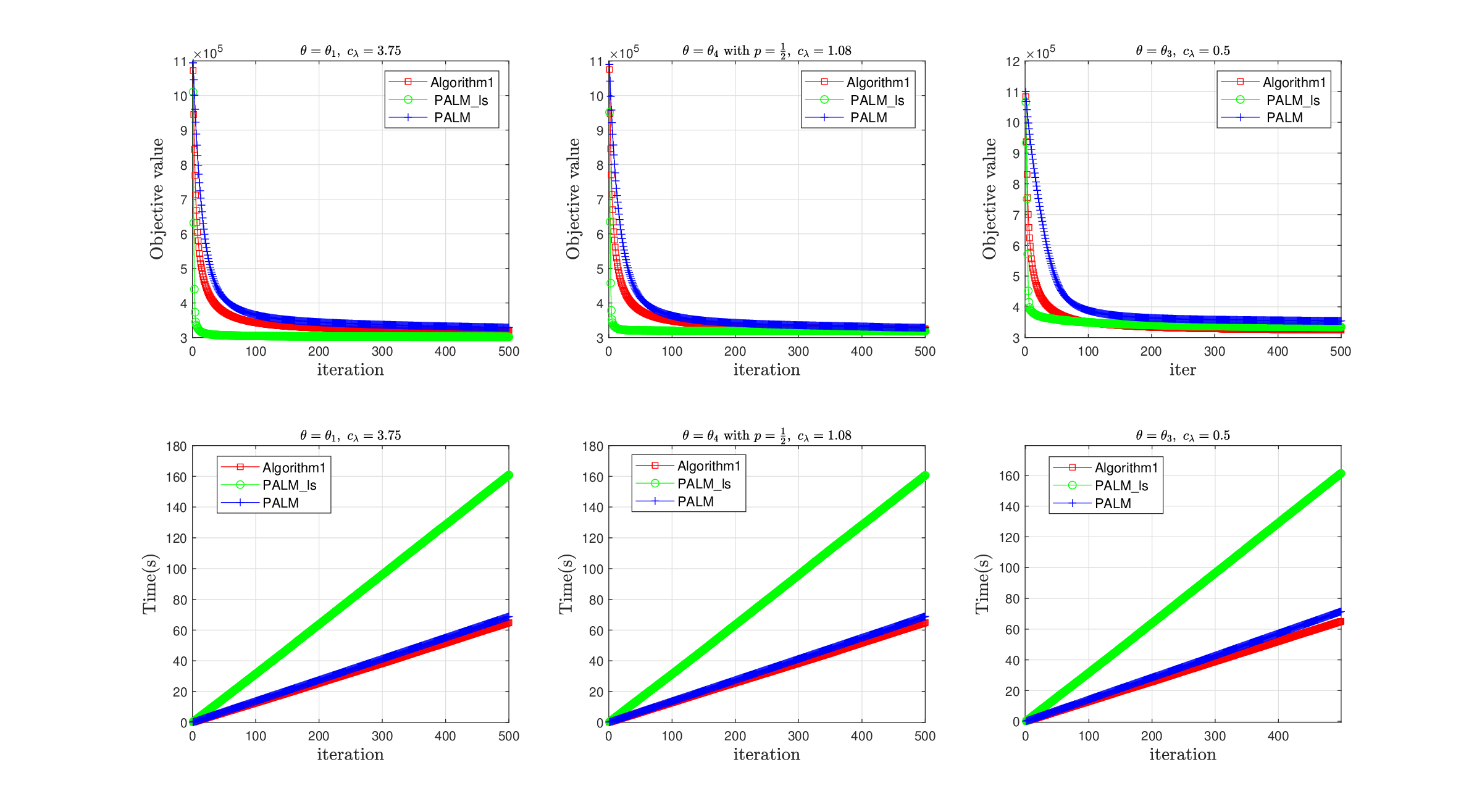}
 \caption{Curves of objective value and running time for the three solvers under Case I with different $\theta$}
  \label{fig01}
\end{figure}

\begin{figure}[h]
 \centering
\includegraphics[width=\textwidth]{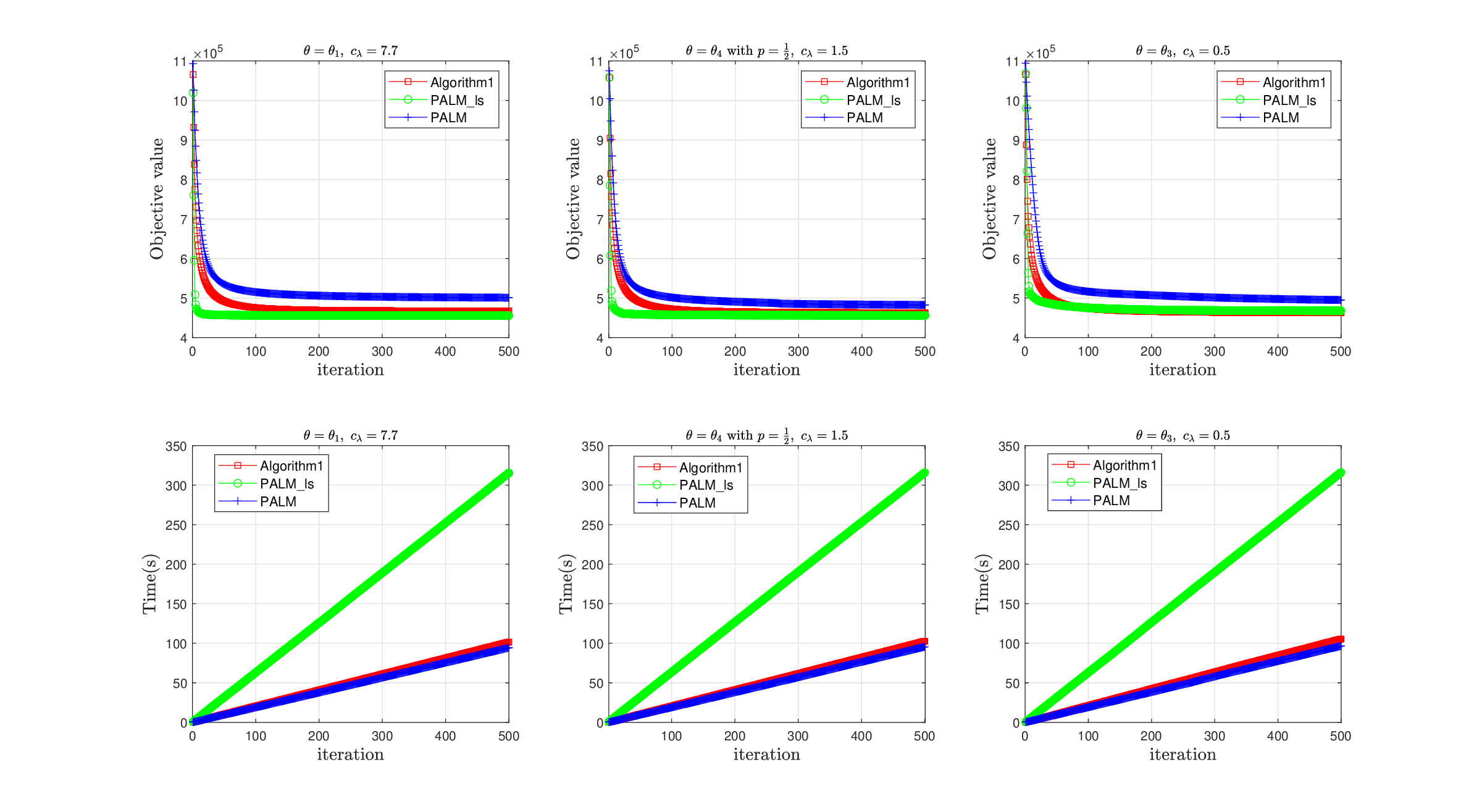}
 \caption{Curves of objective value and running time for the three solvers under Case II with different $\theta$}
  \label{fig02}
\end{figure}

We see that among the three methods, PALM\_ls produces the lowest objective value within the least iterations for the test instances, but its running time increases more quickly as the iterations increases, since the line-search technique requires more function and gradient evaluations. It seems that PALM\_ls exhibits better performance if the maximum number of iterations is used for its stop criterion. However, the stop condition only depending on the maximum number of iterations is not safe for an algorithm. From Figures 2-3, it seems that $50$ is appropriate, but it is unclear whether it is suitable for other classes of test instances. Algorithm \ref{MPAMSC} slightly outperforms PALM in terms of the objective value, and its running time has tiny difference from that of PALM, which coincides with Figure \ref{fig0} because $c_{M^*}=30$ is used for these test instances. 

 
Next we compare the recovery performance of Algorithm \ref{MPAMSC} with that of PALM and Hybrid AMM without line search for solving model \eqref{prob}. The simulated data is generated randomly in the same way as in Section \ref{sec5.3} with $c_{M^*}=30$, and the noisy observation is obtained by \eqref{1bit-model} under Cases I and II with the sample rate ${\rm SR}=0.4$. We adopt the relative error $\frac{\|X^{\rm out}-M^*\|_F}{\|M^*\|_F}$ to evaluate the recovery performance, where $X^{\rm out}=U^{\rm out}(V^{\rm out})^{\top}$ denotes the output of a solver. Consider that the relative errors (REs) generated by PALM\_ls often pushes back, and PALM\_ls requires more running time as shown in Figures \ref{fig01}-\ref{fig02}. We do not compare the performance of Algorithm \ref{MPAMSC} with that of PALM\_ls. Figures \ref{fig1}-\ref{fig2} plot the average RE and rank curves of running five different instances with $r=10r^*$, where the numerical results of the three solvers are obtained under the same stop conditions as in Section \ref{sec5.2} with $\epsilon_1=3\times 10^{-3},\epsilon_2=10^{-3}$ and $k_{\max}=300$.

\begin{figure}[h]
 \centering
\includegraphics[width=\textwidth]{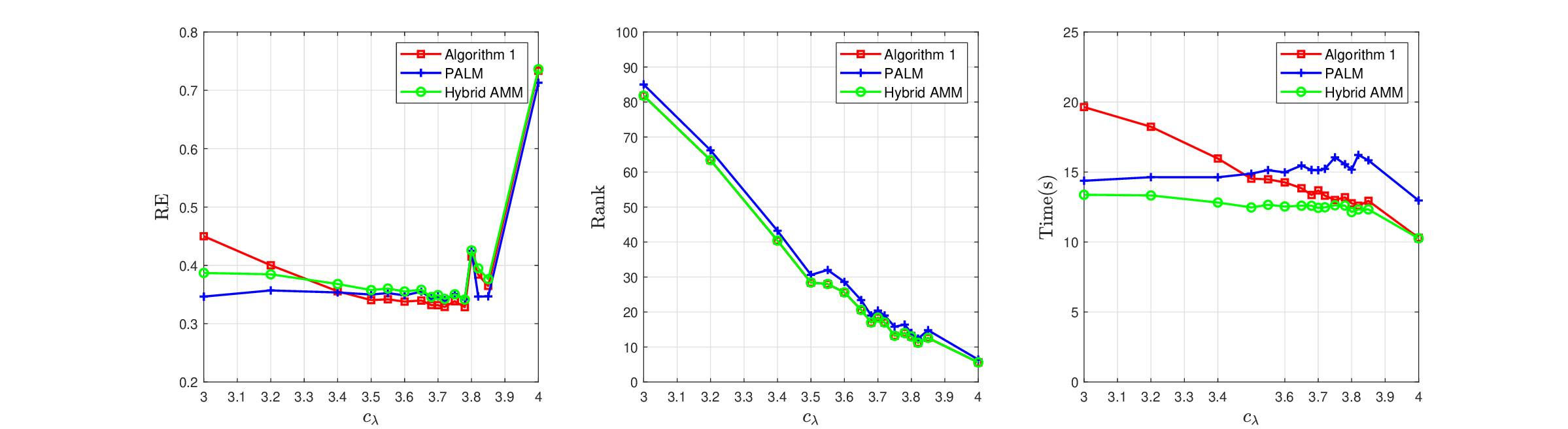}
 \caption{Curves of relative error and rank of the three solvers under Case I with $\theta=\theta_1$}
  \label{fig1}
\end{figure}

\begin{figure}[h]
  \centering
\includegraphics[width=\textwidth]{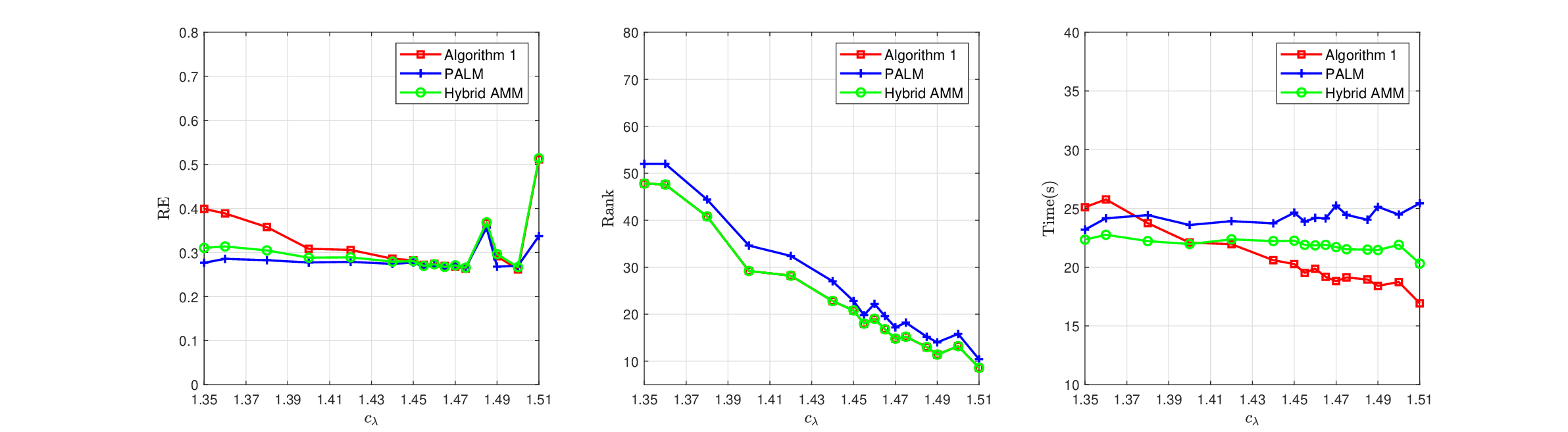}
  \caption{Curves of relative error and rank of the three solvers under Case II with $\theta=\theta_4$ and $p=\frac{1}{2}$}
  \label{fig2}
\end{figure}

 Figure \ref{fig1} shows that Algorithm \ref{MPAMSC} returns sightly lower REs when solving model \eqref{prob} with $\theta=\theta_1$ and $3.4\leq c_{\lambda}<3.8$; while Figure \ref{fig2} indicates that the three solvers yield the comparable REs when solving model \eqref{prob} with $\theta=\theta_4$ for $p=\frac{1}{2}$ and $1.40\leq c_{\lambda}<1.50$. The ranks produced by Algorithm 1 and Hybrid AMM are the same, which are a little lower than those given by PALM. The running time of Algorithm \ref{MPAMSC}  is comparable with that of Hybrid AMM and PALM. The test instances generated with $c_{M^*}=30$ accounts for the tiny difference between the running time of Algorithm \ref{MPAMSC} and that of PALM.



 Table \ref{tab} reports the average RE, rank and time of running five different instances with $n=m$ and $r=10r^*$. The results of Algorithm \ref{MPAMSC}, PALM and Hybrid AMM are all obtained under the same stop condition with $\epsilon_1= 3\times10^{-3}$, $\epsilon_2=10^{-3}$ and $k_{\rm max}=300$. We see that Algorithm \ref{MPAMSC} is slightly superior to PALM and Hybrid AMM in terms of RE for $n=3000$ and $5000$, and requires the comparable running time with Hybrid AMM, which is a little less than that of PALM.
\setlength{\tabcolsep}{1mm}
\begin{table}[h]
	\setlength{\belowcaptionskip}{-0.01cm}
	\setlength\tabcolsep{1.4pt}
	\renewcommand\arraystretch{1.3}
	\centering
	\scriptsize
	\caption{\small The results of three solvers for synthetic data under Case II and $\theta=\theta_4$ with $p=1/2$  }\label{tab}
	
          \begin{tabular*}{\textwidth}{@{\extracolsep{\fill}}cc|lccc|lccc|lccc@{\extracolsep{\fill}}}
			\hline
			& & \multicolumn{4}{l|}{\quad \quad Algorithm 1 }&\multicolumn{4}{l|}{\quad  \quad \ PALM }& \multicolumn{4}{l}{\quad \quad Hybrid AMM}\\
			\hline
		$n$ & $(r^*,SR)$\ &\ $c_{\lambda}$& RE & rank  &time&$c_{\lambda}$&  RE & rank & time&$c_{\lambda}$& RE & rank  &time\\
			\hline	
	1000   &(10,0.30) &1.10& 0.3514 & 15 & 4.88 &1.11 & 0.3448 & 15 & 6.02   &1.10&  0.3460 & 15 & 4.64   \\
		  &(10,0.35)&1.18& 0.3088 & 11 & 5.06 &1.19 & 0.3199 & 11 & 7.16    &1.18&  0.3151 &11 & 5.07 \\
            &(10,0.40)&1.22& 0.2991 & 15 & 5.39 &1.23 & 0.2923 & 15 & 6.68    &1.22& 0.2936 &15 & 5.23 \\
			\hline
      3000  &(10,0.30) &1.49& 0.2893 & 14 & 38.9 &1.50 & 0.2965 & 14 & 45.8   &1.49& 0.2953 & 14 & 38.3  \\
			&(10,0.35) &1.57& 0.2720 & 14 & 42.1 &1.58 & 0.2778 & 15 & 53.7   &1.57& 0.2776 & 14 & 41.8\\
            &(10,0.40) &1.65& 0.2524 & 15 & 43.8 &1.67 & 0.2554 & 14 & 50.6   &1.65& 0.2556 & 15 & 43.8\\
\hline
      5000  &(10,0.30) &1.69&  0.2688 & 17 & 110.1 &1.71 & 0.2703 & 17 & 125.3   &1.69& 0.2715 &17 & 104.5\\
			&(10,0.35) &1.77& 0.2523 & 19 & 114.3 &1.79 & 0.2530 & 18 & 127.9   &1.77& 0.2547 &19 & 111.1\\
            &(10,0.40) &1.87& 0.2358 & 16 & 121.2 &1.89 & 0.2378 & 16 & 132.4   &1.87& 0.2379 &16 & 117.3\\
                
			\hline
		
	\end{tabular*}
\end{table}

 \section{Conclusion}\label{sec6}

 The majorized PAM method in \cite{Hastie15,TaoQianPan22} is extended to a class of low-rank composite factorization model \eqref{prob}, which minimizes alternately the majorization $\widehat{\Phi}_{\lambda,\mu}$ of $\Phi_{\lambda,\mu}$ at each iterate and imposes a subspace correction step on per subproblem to ensure that it has a closed-norm solution. We established the full convergence of the iterate sequence and column subspace sequence of factor pairs under the KL property of $\Phi_{\lambda,\mu}$ and an additional condition that can be satisfied by the column $\ell_{2,0}$-norm function. To the best of our knowledge, this is the first subspace correction AM method with convergence certificate for low-rank factorization models. The obtained results also provide the convergence guarantee for softImpute-ALS proposed in \cite{Hastie15}.


\backmatter

\bmhead{Acknowledgements}
 The first author was supported by the Natural Science Foundation of Guangdong Province under project No.
 2023A1515111167, and the third author was supported by the National Natural Science Foundation of China under project No.12371299.

\section*{Declarations}
{\bf Conflict of interest:}  The authors declare that they have no conflict of interest.

\noindent
{\bf Data availability:} This article only involves the synthetic data. 


\setcounter{equation}{0}
\renewcommand{\theequation}{A. \arabic{equation}}

\begin{appendices}

\section{Supplementary results}\label{secA1}

This part contains two propositions. Among others, Proposition \ref{balance} states that any stationary point $(\overline{U},\overline{V})$ of problem \eqref{prob} satisfies the balance, i.e., $\overline{U}^{\top}\overline{U}=\overline{V}^{\top}\overline{V}$, while Proposition \ref{prop-sdiffgap} provides a condition for inequalities \eqref{ass3-Uk}-\eqref{ass3-Vk} to hold.

  
\renewcommand{\thetheorem}{A.1}

\begin{proposition}\label{balance}
 Let $\mathcal{S}^*$ denote the set of stationary points of \eqref{prob}. If $\theta(t)=\theta_1(t)$ for $t>0$, then $\mathcal{S}^*\subset {\rm crit}\,F_{\mu}$ with $F_{\mu}(U,V)\!:=f(UV^{\top})+(\mu/2)(\|U\|_F^2+\|V\|_F^2)$ for $(U,V)\in\mathbb{X}_r$, and consequently, every $(\overline{U},\overline{V})\in\mathcal{S}^*$
  satisfies $\overline{U}^{\top}\overline{U}=\overline{V}^{\top}\overline{V}$ and $J_{\overline{U}}=J_{\overline{V}}$.
 \end{proposition}
 \begin{proof}
  Pick any $(\overline{U},\overline{V})\in\mathcal{S}^*$. From Definition \ref{def-spoint} and Lemma \ref{subdiff-vtheta}, for any $j\in[r]$,
  \begin{equation}\label{equa1-appendix}
   0\in \nabla f(\overline{U}\overline{V}^{\top})\overline{V}+\mu \overline{U}+ \lambda\partial\vartheta(\overline{U})\ \ {\rm and}\ \
  0\in\big[\nabla f(\overline{U}\overline{V}^{\top})\big]^{\top}\overline{U}
  +\mu\overline{V}+\lambda\partial\vartheta(\overline{V}).
  \end{equation}
  For each $j\in J_{\overline{U}}$, from the first inclusion in \eqref{equa1-appendix} and Lemma \ref{subdiff-vtheta}, we have
  \[
   0=\nabla\!f(\overline{U}\overline{V}^{\mathbb{T}})\overline{V}_{\!j}
     +\mu \overline{U}_j+\lambda\theta'(\|\overline{U}_{\!j}\|){\overline{U}_j}/{\|\overline{U}_{\!j}\|},
 \]
 and hence $\|\nabla f(\overline{U}\overline{V}^{\top})\overline{V}_{\!j}\|=\big(\mu+\lambda\frac{\theta'(\|\overline{U}_j\|)}{\|\overline{U}_j\|}\big)\|\overline{U}_{\!j}\|$. Recall that $\theta'(t)\ge 0$ for $t>0$. For each $j\in J_{\overline{U}}$, it holds 
 $\|\nabla\!f(\overline{U}\overline{V}^{\top})\overline{V}_{\!j}\|>0$ and hence $\overline{V}_{\!j}\neq0$. Consequently, $J_{\overline{U}}\subset J_{\overline{V}}$. For each $j\in J_{\overline{V}}$, from the second inclusion in \eqref{equa1-appendix} and Lemma \ref{subdiff-vtheta},
 \[
  0\in\big[\nabla\!f(\overline{U}\overline{V}^{\top})\big]^{\top}\overline{U}_{\!j}
  +\mu\overline{V}_{\!j}+\lambda\theta'(\|\overline{V}_{\!j}\|){\overline{V}_{\!j}}/{\|\overline{V}_{\!j}\|},
 \]
 which by using the same arguments as above implies  $\overline{U}_{\!j}\ne 0$, and then $J_{\overline{V}}\subset J_{\overline{U}}$. Thus, $J_{\overline{U}}=J_{\overline{V}}:=J$. Together with the above \eqref{equa1-appendix}, it follows that
 \begin{equation*} \nabla\!f(\overline{U}\overline{V}^{\top})\overline{V}_{\!J}+\mu\overline{U}_{\!J}=0\ \ {\rm and}\ \      [\nabla\!f(\overline{U}\overline{V}^{\top})]^{\top}\overline{U}_{\!J}+\mu\overline{V}_{\!J}=0,
 \end{equation*}
 which implies that $\nabla\!F_{\mu}(\overline{U},\overline{V})=0$ and
 $(\overline{U},\overline{V})\in{\rm crit}\,F_{\mu}$. Consequently, the desired inclusion follows. From \cite[Lemma 2.2]{TaoPanBi21},  every $(\overline{U},\overline{V})\in\mathcal{S}^*$ satisfies $\overline{U}^{\top}\overline{U}=\overline{V}^{\top}\overline{V}$. 
 \end{proof}

 To achieve Proposition \ref{prop-sdiffgap} (ii), we need the following technical lemma.
 \renewcommand{\thelemma}{A.2}
 \begin{lemma}\label{lemma-theta}
 Let $D^k$ be a $\overline{r}\times\overline{r}$ diagonal matrix with $\underline{d}\le |D_{ii}^k|\le \overline{d}$ for all $i\in J=[\overline{r}]$. Under Assumptions \ref{ass1}-\ref{ass2}, if $\theta'$ is locally Lipschitz continuous on $\mathbb{R}_{++}$, then there exists a constant $\widehat{c}>0$ such that for each $k\in\mathbb{N}$, 
 \begin{align*}
 \|\nabla\vartheta_{\!J}(\overline{U}_{\!J}^{k+1})-\nabla\vartheta_{\!J}(U_{\!J}^{k+1}D^k)\|_F
  \le\widehat{c}\|\overline{U}_{\!J}^{k+1}-U_{\!J}^{k+1}D^k\|_F,\\
 \|\nabla\vartheta_{\!J}(\overline{V}_{\!J}^{k+1})-\nabla\vartheta_{\!J}(V_{\!J}^{k+1}D^k)\|_F
  \le\widehat{c}\|\overline{V}_{\!J}^{k+1}-V_{\!J}^{k+1}D^k\|_F.  
 \end{align*}
 \end{lemma}
 \begin{proof}
 It suffices to prove that the first inequality holds. 
 By Remark \ref{remark-MPAMSC} (c) and Assumption \ref{ass2}, $\min\{\|U_{j}^{k+1}D_{jj}^k\|,\|V_{j}^{k+1}D_{jj}^k\|\}\ge c_{p}\underline{d}$ for all $j\in J$; while by Corollary \ref{corollary4.1} (iii) $\max\{\|U_{j}^{k+1}D_{jj}^k\|,\|V_{j}^{k+1}D_{jj}^k\|\}\le\overline{\beta}\,\overline{d}$ for all $j\in J$. Note that $\theta'$ is Lipschitz continuous on the interval $[c_p\underline{d},\overline{\beta}\,\overline{d}]$. We denote by $L_{\theta'}$ the Lipschitz constant of $\theta'(\cdot)$ on $[c_p\underline{d},\overline{\beta}\overline{d}]$. 
 For any $u,v\in \mathbb{R}^r$ with $\|u\|,\|v\|\in[\alpha\underline{d},\overline{\beta}\,\overline{d}]$, it holds 
 \begin{align*}
 \big\|\theta'(\|u\|)\frac{u}{\|u\|}-\theta'(\|v\|)\frac{v}{\|v\|}\big\|&=\big\|[\theta'(\|u\|)-\theta'(\|v\|)]\frac{u}{\|u\|}+\theta'(\|v\|)\big(\frac{u}{\|u\|}-\frac{v}{\|v\|}\big)\big\|\nonumber\\   
 &\leq L_{\theta'}\big \|u-v\|+|\theta'(\|v\|)|\Big\|\frac{u-v}{\|u\|}+v\big(\frac{1}{\|u\|}-\frac{1}{\|v\|}\big)\Big\|\nonumber\\
 &\leq L_{\theta'}\big \|u-v\|+\frac{2|\theta'(\|v\|)|}{\|u\|}\|u-v\|.
 \end{align*}
Together with the expression of $\vartheta_{\!J}$ in \eqref{def-vthetaJ} and Lemma \ref{subdiff-vtheta}, it follows that 
 \begin{align*}
 \|\nabla\vartheta_{\!J}(\overline{U}_{\!J}^{k+1})-\nabla\vartheta_{\!J}(U_{\!J}^{k+1}D^k)\|_F
  &\leq \sum_{j=1}^{\overline{r}}\Big[L_{\theta'}+\frac{2|\theta'(\|U_j^{k+1}D_{jj}^k\|)|}{\|\overline{U}_j^{k+1}\|}\Big]\|\overline{U}_j^{k+1}-U_j^{k+1}D_{jj}^k\|\\
  &\le\Big[L_{\theta'}+\frac{2c_0}{\sqrt{\alpha}}\Big] \sum_{j=1}^{\overline{r}}\|\overline{U}_j^{k+1}-U_j^{k+1}D_{jj}^k\|
 \end{align*} 
 with $c_0:=\sup_{t\in[c_p\underline{d},\overline{\beta}\,\overline{d}]}\theta'(t)$. The above inequality implies the desired result. 
 \end{proof}
 
 \renewcommand{\thetheorem}{A.3}
 
 \begin{proposition}\label{prop-sdiffgap}
 Suppose that Assumptions \ref{ass1}-\ref{ass2} hold with $\theta'$ being locally Lipschitz on $\mathbb{R}_{++}$, that there exists $W^*=(\widehat{U}^*,\widehat{V}^*,\overline{U}^*,\overline{V}^*,\widehat{X}^*,\overline{X}^*)\in\mathcal{W}^*$ with $\overline{X}^*$ having distinct nonzero singular values, and that $\Phi_{\lambda,\mu}$ is a KL function. Let $\delta:=\frac{1}{2}\min_{1\le i<j\le\overline{r}+1}[\sigma_i(\overline{X}^*)-\sigma_j(\overline{X}^*)]$ where $\overline{r}$ is the same as in Lemma \ref{Lemma4.1}. Then, 
 \begin{description}
 \item [(i)] for any $X,X'\in\mathbb{B}(\overline{X}^*,\delta/2)\cap\{X\in\mathbb{R}^{n\times m}\,|\ {\rm rank}(X)=\overline{r}\}$, the following inequality
 \[
  \sqrt{\|U_{\!J}{\rm Diag(\omega)}-U_{\!J}'\|_F^2+\|V_{\!J}{\rm Diag(\omega)}-V_{\!J}'\|_F^2}\le\frac{2\sqrt{\overline{r}}}{\delta}\|X-X'\|_F
 \]
 holds for all $(U,V)\in\mathbb{O}^{m,n}(X)$ and $(U',V')\in\mathbb{O}^{m,n}(X')$, where $\omega = (\omega_1,\ldots,\omega_{\overline{r}})^{\top}$ with $\omega_i={\rm sgn}(U_i^\top U'_i+V_i^\top V'_i)$ for each $i\in J:=[\overline{r}]$; 
  
 \item [(ii)] there exists an index $\widetilde{k}\in\mathbb{N}$ such that for all $k\ge\widetilde{k}$, 
 \begin{align*}
 \overline{X}^{k+1},\widehat{X}^{k+1}\in\mathbb{B}(\overline{X}^*,\delta/2)\cap\{X\in\mathbb{R}^{n\times m}\,|\ {\rm rank}(X)\!=\!\overline{r}\},\qquad\\
 {\rm dist}\big(0,\partial \Phi_{\lambda,\mu}(\overline{U}^{{k}+1},\overline{V}^{{k}+1})\big)\le c_{s}(\|\overline{X}^{{k}+1}-\overline{X}^{{k}}\|_F+\|\widehat{X}^{{k}+1}-\overline{X}^{{k}}\|_F)\quad\\
  +c_{s}(\|U^{{k}+1}-\overline{U}^{{k}}\|_F+\|V^{{k}+1}-\!\widehat{V}^{{k}+1}\|_F).
\end{align*}
 \end{description}
 \end{proposition}
 \begin{proof}
 \noindent
{\bf(i)} Fix any $X,X'\in\mathbb{B}(\overline{X}^*,\delta/2)\cap\{X\in\mathbb{R}^{n\times m}\,|\,{\rm rank}(X)=\overline{r}\}$. Pick any $(U,V)\in\mathbb{O}^{m,n}(X)$ and $(U',V')\in\mathbb{O}^{m,n}(X')$. By invoking \cite[Theorem 2.1]{Dopico00} with $A=X$ and $\widetilde{A}=X'$ for $k=1$, for each $i\in[\overline{r}]$, with $\delta_i=\min_{j\in[n]\backslash\{i\}}|\sigma_i(X)-\sigma_j(X')|$ it holds
\begin{equation}\label{UVi-result}
 \sqrt{\|U_i\omega_i-U_i'\|^2+\|V_i\omega_i-V_i'\|^2}\le\frac{2}{\delta_i}\|X-X'\|_F.   
\end{equation}
Fix any $i\in J$. From \cite[IV Theorem 4.11]{Steward90} and the definition of $\delta$, for each $j\in[n]\backslash\{i\}$,  
\begin{align*}
|\sigma_i(X)-\sigma_j(X')|&=|\sigma_i(X)-\sigma_i(\overline{X}^*)+\sigma_i(\overline{X}^*)-\sigma_j(\overline{X}^*)+\sigma_j(\overline{X}^*)-\sigma_j(X')|\\
&\ge |\sigma_i(\overline{X}^*)-\sigma_j(\overline{X}^*)|-|\sigma_i(X)-\sigma_i(\overline{X}^*)|-|\sigma_j(X')-\sigma_j(\overline{X}^*)|\ge\delta,
\end{align*}
which implies that $\delta_i\ge \delta$ for each $i\in[\overline{r}]$. Together with the above \eqref{UVi-result}, it follows 
\begin{align*}
&\sqrt{\|U_J{\rm Diag(\omega)}-U'_J\|_F^2+\|V_J{\rm Diag(\omega)}-V'_J\|_F^2}\\
&\le\sqrt{\overline{r}}\max_{i\in J}\sqrt{\|U_i\omega_i-U_i'\|^2+\|V_i\omega_i-V_i'\|^2}\le\frac{2\sqrt{\overline{r}}}{\delta}\|X-X'\|_F.
\end{align*}

\noindent
{\bf(ii)} According to Corollary \ref{corollary4.1} (iii), it suffices to consider that $\Phi_{\lambda,\mu}(\overline{U}^k,\overline{V}^k)>\varpi^*$ for all $k\in\overline{k}$, where $\overline{k}$ is the same as in Proposition \ref{prop-nzind}. If not, we will have $\overline{X}^k=\overline{X}^{\overline{k}}$ for all $k>\overline{k}$, and consequently, the sequence $\{\overline{X}^k\}_{k\in\mathbb{N}}$ is convergent. Write
 \[
 \Omega^*\!:=\!\big\{(U,V)\in\mathbb{X}_r\,|\,\exists (\widehat{U},\widehat{V})\in\!\mathbb{X}_r, \widehat{X},\overline{X}\in\mathbb{R}^{n\times m}\ {\rm s.t.}\ (\widehat{U},\widehat{V},U,V,\widehat{X},\overline{X})\in\mathcal{W}^*\big\}. 
 \]
 By Theorem \ref{Sub-convergence}, $\mathcal{W}^*$ is nonempty and compact, so is the set $\Omega^*$. Moreover, $\Phi_{\lambda,\mu}(U,V)=\varpi^*$ for all $(U,V)\in\Omega^*$. From the KL property of $\Phi_{\lambda,\mu}$ on $\Omega^*$ and \cite[Lemma 6]{Bolte14}, there exist $\varepsilon>0,\eta>0$ and $\varphi\in\Upsilon_{\!\eta}$ such that for all $(U,V)\in[\varpi^*<\Phi_{\lambda,\mu}<\varpi^*+\eta]\cap\mathfrak{B}(\Omega^*,\varepsilon)$ with $\mathfrak{B}(\Omega^*,\varepsilon)\!:=\big\{(U,V)\in\mathbb{X}_{r}\,|\,{\rm dist}((U,V),\Omega^*)\le\varepsilon\big\}$,
 \[
  \varphi'(\Phi_{\lambda,\mu}(U,V)-\varpi^*){\rm dist}(0,\partial\Phi_{\lambda,\mu}(U,V))\ge 1.
 \]
 By Theorems \ref{converge-obj}-\ref{Sub-convergence}, if necessary by increasing $\overline{k}$, where $\overline{k}$ is the same as in Proposition \ref{prop-nzind}, we have  $(\overline{U}^{k},\overline{V}^{k})\in[\varpi^*<\Phi_{\lambda,\mu}<\varpi^*+\eta]\cap\mathfrak{B}(\Omega^*,\varepsilon)$ for all $k\ge\overline{k}$. Consequently, 
 \[
  \varphi'\big(\Phi_{\lambda,\mu}(\overline{U}^{k},\overline{V}^{k})-\varpi^*\big){\rm dist}(0,\partial\Phi_{\lambda,\mu}(\overline{U}^{k},\overline{V}^{k}))\ge 1\quad\forall k\ge\overline{k}.
 \]
 For each $k\ge\overline{k}$, let $\Gamma_{k,k+1}\!:=\varphi\big(\Phi_{\lambda,\mu}(\overline{U}^k,\overline{V}^k)\!-\!\varpi^*\big)
 \!-\!\varphi\big(\Phi_{\lambda,\mu}(\overline{U}^{k+1},\overline{V}^{k+1})\!-\!\varpi^*\big)$. From the above inequality and the concavity of $\varphi$, it follows that for all $k\ge\overline{k}$, 
 \begin{align}\label{ineq1-KLPhi}
 \Phi_{\lambda,\mu}(\overline{U}^{k},\overline{V}^{k})-\Phi_{\lambda,\mu}(\overline{U}^{k+1},\overline{V}^{k+1})
 &\le{\rm dist}(0,\partial \Phi_{\lambda,\mu}(\overline{U}^{k},\overline{V}^{k}))\,\Gamma_{k,k+1}.
 \end{align}  
 In the rest of the proof, let $\widehat{\omega}^{k+1}=(\widehat{\omega}_1^{k+1},\ldots,\widehat{\omega}_{\overline{r}}^{k+1})^{\top}$ and $\widetilde{\omega}^{k}=(\widetilde{\omega}_1^{k+1},\ldots,\widetilde{\omega}_{\overline{r}}^{k})^{\top}$ with 
 $\widehat{\omega}_i^{k+1}={\rm sgn}((\widehat{P}_i^{k+1})^\top\widetilde{R}_i^{k+1}\!+(\widetilde{Q}_i^{k+1})^\top\overline{P}_i^{k+1})$ and
 $\widetilde{\omega}_i^{k}={\rm sgn}((\widetilde{R}_i^{k})^\top\widetilde{R}_i^{k+1}\!+(\overline{P}_i^{k})^\top\overline{P}_i^{k+1})$ for each $k$, where $\widehat{P}^{k+1},\widetilde{R}^{k+1},\overline{P}^k$ and $\overline{P}^{k+1}$ are the same as in the proof of Lemma \ref{lemma-UVk}.

Since $\overline{X}^*$ is a cluster point of $\{\widehat{X}^k\}$ and $\{\overline{X}^k\}$ by Theorem \ref{Sub-convergence}, in view of Lemma \ref{rank-prop} (i), there exists $\widetilde{k}\ge\overline{k}$ such that $\overline{X}^{\widetilde{k}+1},\widehat{X}^{\widetilde{k}+1}\in\mathbb{B}(\overline{X}^*,\delta/4)\cap\{X\in\mathbb{R}^{n\times m}\,|\,{\rm rank}(X)=\overline{r}\}$. From item (i) with $(X,X')=(\overline{X}^{\widetilde{k}+1},\widehat{X}^{\widetilde{k}+1})$ and $(X,X')=(\overline{X}^{\widetilde{k}},\overline{X}^{\widetilde{k}+1})$, we conclude that \eqref{dist-X1}-\eqref{dist-X2} hold with $k=\widetilde{k}, 
R_1^{k+1}={\rm Diag}(\widehat{\omega}^{k+1})$ and $ R_2^{k}={\rm Diag}(\widetilde{\omega}^{k})$ but $\frac{2}{\sigma_{\overline{r}}(\overline{X}^{k+1})}$ replaced by $\frac{2\sqrt{\overline{r}}}{\delta}$. From the proof of Lemma \ref{lemma-UVk}, its conclusion still holds with $k=\widetilde{k},A_1^{k+1}=R_1^{k+1}, B_1^{k}=R_2^{k}$ but $\frac{\sqrt{\alpha}+4\overline{\beta}}{2\alpha}$ replaced by $\frac{\delta+4\overline{\beta}\sqrt{\overline{r}\alpha}}{2\sqrt{\alpha}\delta}$. Let $J=[\overline{r}]$. From the expression of $\vartheta_{\!J}$ in \eqref{def-vthetaJ}, for all $k\in\mathbb{N}$,  
\begin{align}\label{ineq1-A1B1k}
 \nabla\vartheta_{\!J}(U_{\!J}^{k+1})B_1^k=\nabla\vartheta_{\!J}(U_{\!J}^{k+1}){\rm Diag}(\widetilde{\omega}^{k}))=\nabla\vartheta_{\!J}(U^{k+1}_{\!J}{\rm Diag}(\widetilde{\omega}^{k})),\qquad\\ 
 \label{ineq2-A1B1k}
  \nabla\vartheta_{\!J}(V_{\!J}^{k+1})A_1^{k+1}=\nabla\vartheta_{\!J}(V_{\!J}^{k+1}){\rm Diag}(\widehat{\omega}^{k+1}))=\nabla\vartheta_{\!J}(V^{k+1}_{\!J}{\rm Diag}(\widehat{\omega}^{k+1})).
\end{align}
Then, in view of the local Lipschitz continuity of $\theta'$ on $\mathbb{R}_{++}$, using Lemma \ref{lemma-theta} and equations \eqref{ineq1-A1B1k}-\eqref{ineq2-A1B1k} leads to inequalities \eqref{ass3-Uk}-\eqref{ass3-Vk} for  $k=\widetilde{k}$. By Proposition \ref{subdiff-gap}, 
\begin{align}\label{bound-distuvk}
{\rm dist}\big(0,\partial \Phi_{\lambda,\mu}(\overline{U}^{\widetilde{k}+1},\overline{V}^{\widetilde{k}+1})\big)&\le c_{s}(\|\overline{X}^{\widetilde{k}+1}-\overline{X}^{\widetilde{k}}\|_F+\|\widehat{X}^{\widetilde{k}+1}-\overline{X}^{\widetilde{k}}\|_F)\nonumber\\
&\quad +c_{s}(\|U^{\widetilde{k}+1}-\overline{U}^{\widetilde{k}}\|_F+\|V^{\widetilde{k}+1}-\widehat{V}^{\widetilde{k}+1}\|_F).
\end{align}
This shows that the conclusion holds for $k=\widetilde{k}$. 
Write $a:=\min(\frac{\overline{\gamma}}{4},\frac{\overline{\gamma}}{16\overline{\beta}^2})$, and let
\[
  \Xi^{k}:=\|\overline{X}^{k}\!-\!\overline{X}^{k-1}\|_F+\|\widehat{X}^{k}\!-\!\overline{X}^{k-1}\|_F+\|U^{k}\!-\overline{U}^{k-1}\|_F+\|V^{k}-\widehat{V}^{k}\|_F
\]
for each $k\in\mathbb{N}$. In what follows, we prove by induction that for all $\nu\ge\widetilde{k}+2$,    
\begin{align}\label{aim-ineq1}
 \max\{\|\overline{X}^{\nu}-\overline{X}^*\|,\|\widehat{X}^{\nu}-\overline{X}^*\|\}\le\delta/2,\qquad\\
  \label{aim-ineq2}
  \frac{\sqrt{a}}{2}\sum_{j=\widetilde{k}+2}^{\nu}\Xi^j\le\frac{\sqrt{a}}{4}\sum_{j=\widetilde{k}+2}^{\nu}\Xi^{j-1}+\frac{c_s}{\sqrt{a}}\sum_{j=\widetilde{k}+2}^{\nu}\Gamma_{j-1,j}.
\end{align}
Indeed, from Corollary \ref{corollary4.1} (iii) and the expression of $\Xi^{\widetilde{k}+2}$, it follows that 
\begin{align}\label{ineq-ktilde}
 \frac{\sqrt{a}}{2}\Xi^{\widetilde{k}+2}&\le\sqrt{\Phi_{\lambda,\mu}(\overline{U}^{\widetilde{k}+1},\overline{V}^{\widetilde{k}+1})-\Phi_{\lambda,\mu}(\overline{U}^{\widetilde{k}+2},\overline{V}^{\widetilde{k}+2})}\nonumber\\
 &\stackrel{\eqref{ineq1-KLPhi}}{\le}\sqrt{{\rm dist}(0,\partial \Phi_{\lambda,\mu}(\overline{U}^{\widetilde{k}+1},\overline{V}^{\widetilde{k}+1}))\Gamma_{\widetilde{k}+1,\widetilde{k}+2}}\stackrel{\eqref{bound-distuvk}}{\le}\sqrt{c_{s}\Xi^{\widetilde{k}+1}\Gamma_{\widetilde{k}+1,\widetilde{k}+2}}\nonumber\\
 &\le \frac{\sqrt{a}}{4}\Xi^{\widetilde{k}+1}+\frac{c_s}{\sqrt{a}}\Gamma_{\widetilde{k}+1,\widetilde{k}+2}.
 \end{align}
 By Corollary \ref{corollary4.1} (i) and (iv) and Theorem \ref{converge-obj}, for
 all $k\ge\widetilde{k}$ (if necessary by increasing $\widetilde{k}$), we have 
 $\Xi^{\widetilde{k}+1} +\frac{4c_{s}}{a}\varphi\big(\Phi_{\lambda,\mu}(\overline{U}^{\widetilde{k}+1},\overline{V}^{\widetilde{k}+1})-\varpi^*\big)\le\frac{1}{2}\delta$, which by \eqref{ineq-ktilde} and the nonnegativity of $\varphi$ implies  
 $\Xi^{\widetilde{k}+2}\le\delta/4$. Then, from   $\overline{X}^{\widetilde{k}+1},\widehat{X}^{\widetilde{k}+1}\in\mathbb{B}(\overline{X}^*,\delta/4)$, we get
 \begin{align*}
  \|\overline{X}^{\widetilde{k}+2}-\overline{X}^*\|\le\|\overline{X}^{\widetilde{k}+1}-\overline{X}^*\|+\|\overline{X}^{\widetilde{k}+2}-\overline{X}^{\widetilde{k}+1}\|\le\|\overline{X}^{\widetilde{k}+1}-\overline{X}^*\|+\Xi^{\widetilde{k}+2}\le\delta/2,\\
  \|\widehat{X}^{\widetilde{k}+2}-\overline{X}^*\|\le\|\overline{X}^{\widetilde{k}+1}-\overline{X}^*\|+\|\widehat{X}^{\widetilde{k}+2}-\overline{X}^{\widetilde{k}+1}\|\le\|\overline{X}^{\widetilde{k}+1}-\overline{X}^*\|+\Xi^{\widetilde{k}+2}\le\delta/2.
 \end{align*}
 Along with Lemma \ref{rank-prop} (i),  $\overline{X}^{\widetilde{k}+2},\widehat{X}^{\widetilde{k}+2}\in\{X\in\!\mathbb{R}^{n\times m}\,|\,{\rm rank}(X)=\overline{r}\}\cap\mathbb{B}(\overline{X}^*,\delta/2)$. This, together with the above \eqref{ineq-ktilde}, shows that inequalities \eqref{aim-ineq1}-\eqref{aim-ineq2} hold for $\nu=\widetilde{k}+2$. Now assuming that inequalities \eqref{aim-ineq1}-\eqref{aim-ineq2} hold for some $\nu\ge \widetilde{k}+2$, we claim that they hold for $\nu+1$. Since \eqref{aim-ineq1} holds, from item (i) with 
 $(X,X')=(\overline{X}^{\nu},\widehat{X}^{\nu})$ and $(X,X')=(\overline{X}^{\nu-1},\overline{X}^{\nu})$, inequalities \eqref{dist-X1}-\eqref{dist-X2} hold with $k=\nu,R_1^{k}={\rm Diag}(\widehat{\omega}^{k}),R_2^{k-1}={\rm Diag}(\widehat{\omega}^{k-1})$ but $\frac{2}{\sigma_{\overline{r}}(\overline{X}^{k+1})}$ replaced by $\frac{2\sqrt{\overline{r}}}{\delta}$.  By the proof of Lemma \ref{lemma-UVk}, its conclusion still holds with $k=\nu,A_1^{k}=R^{k}_1$ and $B_1^{k-1}=R^{k-1}_2$ but $\frac{\sqrt{\alpha}+4\overline{\beta}}{2\alpha}$ replaced by $\frac{\delta+4\overline{\beta}\sqrt{\alpha\overline{r}}}{2\sqrt{\alpha}\delta}$. In addition, by the local Lipschitz continuity of $\theta'$ on $\mathbb{R}_{++}$, using Lemma \ref{lemma-theta} and equations \eqref{ineq1-A1B1k}-\eqref{ineq2-A1B1k} leads to inequalities \eqref{ass3-Uk}-\eqref{ass3-Vk} for $k=\nu$. Then, from Proposition \ref{subdiff-gap},
\begin{align}\label{bound-distuvk2}
{\rm dist}\big(0,\partial \Phi_{\lambda,\mu}(\overline{U}^{\nu},\overline{V}^{\nu})\big)&\le c_{s}(\|\overline{X}^{\nu}-\overline{X}^{\nu-1}\|_F+\|\widehat{X}^{\nu}-\overline{X}^{\nu-1}\|_F)\nonumber\\
&\quad +c_{s}(\|U^{\nu}-\overline{U}^{\nu-1}\|_F+\|V^{\nu}-\widehat{V}^{\nu-1}\|_F).
\end{align} 
From Corollary \ref{corollary4.1} (iii), the expression of $\Xi^{\nu}$ and $a=\min(\frac{\overline{\gamma}}{4},\frac{\overline{\gamma}}{16\overline{\beta}^2})$, it follows that 
\begin{align*}
 \frac{\sqrt{a}}{2}\Xi^{\nu+1}&\le\sqrt{\Phi_{\lambda,\mu}(\overline{U}^{\nu},\overline{V}^{\nu})-\Phi_{\lambda,\mu}(\overline{U}^{\nu+1},\overline{V}^{\nu+1})}
 \stackrel{\eqref{ineq1-KLPhi}}{\le}\sqrt{{\rm dist}(0,\partial \Phi_{\lambda,\mu}(\overline{U}^{\nu},\overline{V}^{\nu}))\Gamma_{\nu,\nu+1}}\\
 &\stackrel{\eqref{bound-distuvk2}}{\le}\sqrt{c_{s}\Xi^{\nu}\Gamma_{\nu,\nu+1}}\le \frac{\sqrt{a}}{4}\Xi^{\nu}+\frac{c_s}{\sqrt{a}}\Gamma_{\nu,\nu+1}.
 \end{align*}
 Together with $\sum_{j=\widetilde{k}+2}^{\nu+1}\Xi^j
 =\Xi^{\nu+1}+\sum_{j=\widetilde{k}+2}^{\nu}\Xi^j$ and inequality \eqref{aim-ineq2}, we have
 \begin{align*}
 \frac{\sqrt{a}}{2}\sum_{j=\widetilde{k}+2}^{\nu+1}\Xi^j
  &\le\frac{\sqrt{a}}{4}\Xi^{\nu}+\frac{c_s}{\sqrt{a}}\Gamma_{\nu,\nu+1}+\frac{\sqrt{a}}{4}\sum_{j=\widetilde{k}+2}^{\nu}\Xi^{j-1}+\frac{c_s}{\sqrt{a}}\sum_{j=\widetilde{k}+2}^{\nu}\Gamma_{j-1,j}\\
 &=\frac{\sqrt{a}}{4}\sum_{j=\widetilde{k}+2}^{\nu+1}\Xi^{j-1}+\frac{c_s}{\sqrt{a}}\sum_{j=\widetilde{k}+2}^{\nu+1}\Gamma_{j-1,j}.
\end{align*}
This shows that \eqref{aim-ineq2} holds for $\nu+1$. Also, from \eqref{aim-ineq2} and the nonnegativity of $\varphi$, we have
\(
  \sum_{j=\widetilde{k}+2}^{\nu}\Xi^{j}\le\frac{1}{2}\sum_{j=\widetilde{k}+2}^{\nu}\Xi^{j-1}+\frac{2c_{s}}{a}\varphi\big(\Phi_{\lambda,\mu}(\overline{U}^{\widetilde{k}+2},\overline{V}^{\widetilde{k}+2})\!-\!\varpi^*\big),
\)
and consequently, 
\[
  \frac{1}{2}\sum_{j=\widetilde{k}+2}^{\nu}\Xi^{j-1}+\Xi^{\nu} \le \Xi^{\widetilde{k}+1}+\frac{2c_{s}}{a}\varphi\big(\Phi_{\lambda,\mu}(\overline{U}^{\widetilde{k}+2},\overline{V}^{\widetilde{k}+2})\!-\!\varpi^*\big).
\]
This, by Corollary \ref{corollary4.1} (i) and (iv) and Theorem \ref{converge-obj}, implies  $\sum_{j=\widetilde{k}+2}^{\nu}\Xi^j\le\delta/8$ (if necessary by increasing $\widetilde{k}$). Consequently,
\begin{align*}
  \|\overline{X}^{\nu+1}-\overline{X}^*\|
  &\le\|\overline{X}^{\widetilde{k}+1}-\overline{X}^*\|+\sum_{j=\widetilde{k}+2}^{\nu}\|\overline{X}^{j+1}-\overline{X}^{j}\|\le\|\overline{X}^{\widetilde{k}+1}-\overline{X}^*\|+\sum_{j=\widetilde{k}+2}^{\nu}\Xi^j\le\delta/2,\\
  \|\widehat{X}^{\nu+1}-\overline{X}^*\|
  &\le\|\widehat{X}^{\widetilde{k}+1}-\overline{X}^*\|+\sum_{j=\widetilde{k}+2}^{\nu}\big[\|\widehat{X}^{j+1}-\overline{X}^{j}\|+\|\overline{X}^{j}-\overline{X}^{j-1}\|+\|\widehat{X}^{j}-\overline{X}^{j-1}\|\big]\\
  &\le\|\widehat{X}^{\widetilde{k}+1}-\overline{X}^*\|+\sum_{j=\widetilde{k}+2}^{\nu}\big[\Xi^j+\Xi^{j-1}\big]\le\delta/4+\delta/4=\delta/2.
 \end{align*}
 The above two inequalities show that \eqref{aim-ineq1} holds for $\nu+1$. Thus, \eqref{aim-ineq1}-\eqref{aim-ineq2} hold for all $k\ge\widetilde{k}+2$. Since inequality \eqref{aim-ineq1} holds for $\nu+1$, using item (i) with $X=\overline{X}^{\nu+1}$ and $X'=\widehat{X}^{\nu+1}$ and following the same arguments as above, we have  
\begin{align*}\label{bound-distuvk3}
{\rm dist}\big(0,\partial \Phi_{\lambda,\mu}(\overline{U}^{\nu+1},\overline{V}^{\nu+1})\big)&\le c_{s}(\|\overline{X}^{\nu+1}-\overline{X}^{\nu}\|_F+\|\widehat{X}^{\nu+1}-\overline{X}^{\nu}\|_F)\nonumber\\
&\quad +c_{s}(\|U^{\nu+1}-\overline{U}^{\nu}\|_F+\|V^{\nu+1}-\widehat{V}^{\nu}\|_F).
\end{align*}  
This along with \eqref{aim-ineq1} shows that the desired conclusion holds.
\end{proof}
 \section{A line-search PALM method for \eqref{prob}}\label{secB}

 As mentioned in the introduction, the iterations of the PALM methods in \cite{Bolte14,Pock16,XuYin17} depend on the Lipschitz constants of $\nabla_{U}F(\cdot,V^k)$ and $\nabla_{V}F(U^{k+1},\cdot)$. An immediate upper estimation for them is $L_{f}\max\{\|V^k\|^2,\|U^{k+1}\|^2\}$, but it is too large and will make the performance of PALM methods worse. Here we present a PALM method by searching for tighter upper estimations for them. Fix any $(U,V)\in\mathbb{X}_r$ and $\tau>0$. For any $(U',V')\in\mathbb{X}_r$, define $\mathcal{Q}_{U}(U',V';U,V,\tau)$ and $\mathcal{Q}_{V}(U',V';U,V,\tau)$ by
 \begin{align*}
   \mathcal{Q}_{U}(U', V';U,V,\tau)\!:=F(U,V)\!+\!\langle\nabla_{U}F(U,V),U'\!-\!U\rangle\!+\!\frac{\tau}{2}\|U'\!-\!U\|_F^2\!+\!\frac{\mu}{2}\|U'\|^2_F\!+\!\lambda \vartheta(U'),\\
   \mathcal{Q}_{V}(U',V';U,V,\tau)\!:=\!F(U,V)\!+\!\langle\nabla_{V}F(U,V),V'\!-\!V\rangle\!+\!\frac{\tau}{2}\|V'\!-\!V\|_F^2\!+\!\frac{\mu}{2}\|V'\|^2_F\!+\!\lambda \vartheta(V').
 \end{align*}
 The iterations of the line-search PALM method are described as follows.
 \begin{algorithm}[h]
 \caption{\label{LSPALM}{\bf(A line-search PALM method)}}
 \begin{algorithmic} 
 \State{Input: $\varrho_1>1,\varrho_2>1$ and $0<\underline{\alpha}<\overline{\alpha}$.
 Choose and an initial point $(U^0,V^0)\in\mathbb{X}_r$.}
 \State{\textbf{For}  $k=0,1,2,\ldots$  \textbf{do}}
 \State{\qquad\  1. Select $\alpha_k\in[\underline{\alpha},\overline{\alpha}]$ and compute
    \[U^{k+1}\in\mathop{\arg\min}_{U\in\mathbb{R}^{n\times r}}\mathcal{Q}_{U}(U,V^k;U^k,V^k,\alpha_k).\]}
\State{\qquad\ 2. \textbf{While} {  $F(U^{k+1},U^k)>\mathcal{Q}_{U}(U^{k+1},V^k;U^k,V^k,\alpha_k)$}  \textbf{do}}
    \State{\qquad\qquad (i) $\alpha_{k}=\varrho_1\alpha_k$;}

          \State{\qquad\qquad (ii)  Compute $U^{k+1}\in\mathop{\arg\min}_{U\in\mathbb{R}^{n\times r}}\mathcal{Q}_{U}(U,V^k;U^k,V^k,\alpha_k)$
                }
   \State{\qquad\quad   \ \textbf{end(While)}}
\State{\qquad\ 3. Select $\alpha_k\in[\underline{\alpha},\overline{\alpha}]$.
  Compute
    \[V^{k+1}\in\!\mathop{\arg\min}_{V\in\mathbb{R}^{m\times r}}\mathcal{Q}_{V}(U^{k+1},V;U^{k+1},V^k,\alpha_k).\] }

   \State{\quad\quad4. \textbf{While} { $F(U^{k+1},V^{k+1})>\mathcal{Q}_{V}(U^{k+1},V^{k+1};U^{k+1},V^k,\alpha_k)$} \textbf{do}}
    \State{\qquad\qquad (i) $\alpha_{k}=\varrho_2\alpha_k$;}
          \State{\qquad\qquad (ii)  Compute $V^{k+1}\in\mathop{\arg\min}_{V\in\mathbb{R}^{m\times r}}\mathcal{Q}_{V}(U^{k+1},V;U^{k+1},V^k,\alpha_k)$. }
   \State{\qquad\quad   \ \textbf{end(While)}}
   \State{\textbf{end}}
 \end{algorithmic}
 \end{algorithm}

 \begin{remark}
  In the implementation of Algorithm \ref{LSPALM}, we use the Barzilai-Borwein (BB) rule \cite{Barzilai88} to capture the initial $\alpha_k$ in steps 1 and 3; for example, for $\alpha_k$ in step 1, we use
  \begin{equation*}
  \min\Big\{\max\Big\{ 		 \frac{\|\langle\nabla_{U}F(U^{k},V^k)-\nabla_{U}F(U^{k-1},V^k)\|_F}{\|U^{k}-U^{k-1}\|_F}\Big),\underline{\alpha}\Big\},\overline{\alpha}\Big\}.
 \end{equation*}
 \end{remark}




\end{appendices}



\end{document}